\documentclass[EJP]{ejpecp}

\usepackage[utf8]{inputenc}
\usepackage{stmaryrd}
\usepackage{amssymb}
\usepackage{tikz}
\usepackage{xcolor}
\usetikzlibrary{arrows,decorations.markings}

\SHORTTITLE{A branching random walk among disasters}

\TITLE{A branching random walk among disasters}

\AUTHORS{
  Nina Gantert\footnote{Technical University of Munich, Germany.
    \EMAIL{gantert@ma.tum.de, junk@tum.de}}
  \and 
  Stefan Junk\footnotemark[1]
  }


\AMSSUBJ{60K37; 60J80; 82D60} 

\SUBMITTED{August 8, 2016} 
\ACCEPTED{June 12, 2017} 

\ARXIVID{1608.02440} 

\VOLUME{22}
\YEAR{2017}
\PAPERNUM{67}
\DOI{10.1214/17-EJP75}

\ABSTRACT{We consider a branching random walk in a random space-time environment of disasters where each particle is killed when meeting a disaster. This extends the model of the \lq\lq random walk in a disastrous random environment\rq\rq\ introduced by \cite{shiga}. We obtain a criterion for positive survival probability, see Theorem \ref{thm:main}. 

The proofs for the subcritical and the supercritical cases follow standard arguments, which involve moment methods and a comparison with an embedded branching process with i.i.d.\ offspring distributions. However, for this comparison we need to show that the survival rate of a single particle equals the survival rate of a single particle returning to the origin (Proposition \ref{prop:propSTilde}). We prove this statement by making use of stochastic domination.

The proof of almost sure extinction in the critical case is more difficult and uses the techniques from \cite{garetmarchand}, going back to \cite{criticalcontact}. We also show that, in the case of survival, the number of particles grows exponentially fast. }

\newcommand{\R}{\mathbb R}
\newcommand{\Q}{\mathbb Q}
\newcommand{\Z}{\mathbb Z}
\newcommand{\N}{\mathbb N}
\newcommand{\E}{\mathbb E}
\renewcommand{\P}{\mathbb P}
\renewcommand{\O}{\mathbb F}

\newcommand{\eps}{\varepsilon}
\newcommand{\cadlag}{càdlàg }

\newcommand{\B}{\mathbb B}
\newcommand{\T}{\mathbb T}
\newcommand{\F}{\mathbb F}

\newcommand{\1}{\mathbb 1}

\newcommand{\I}{\mathcal I}

\newcommand{\dd}{\mathrm{d}}
\newcommand{\sign}{\text{sign}\;}
\newcommand{\ox}{\omega^{(x)}}
\newcommand{\oy}{\omega^{(y)}}

\makeatletter
\providecommand*{\cupdot}{\mathbin{\mathpalette\@cupdot{}}}
\newcommand*{\@cupdot}[2]{\ooalign{$\m@th#1\cup$\cr\hidewidth$\m@th#1\cdot$\hidewidth}}
\makeatother

\newcommand*\goodoverline[1]{\overline{#1\vphantom{+1}}}

\newcommand{\lobar}{\pi^{-1}(\widetilde \omega)}

\newcommand{\n}{\llbracket N\rrbracket}
\newcommand{\isd}{\stackrel{d}{=}}

\newtheorem*{ttheorem}{Theorem}

\newcommand{\U}{\mathcal U}

\newcommand{\m}{\llbracket m\rrbracket}

\begin{document}


\section{Introduction}\label{sec:intro}


In this work we introduce a branching random walk on $\Z^d$ in a killing random environment. The process consists of particles performing a branching random walk in continuous time. All particles jump independently at rate $\kappa$ and give birth to children at rate $\lambda$. The jump rate $\kappa$, the birth rate $\lambda$ and the distribution $q$ of the number of children do not change over time and space, and are the parameters of the model. 

We then consider this process in a random environment $\omega$ given by disasters in space-time, defined as follows: The environment $\omega$ consists of a collection $\big(\ox\big)_{x \in \Z^d}$ of i.i.d.\ random variables where $\ox=(\ox(t))_{t \geq 0}$ is a Poisson process of rate one. 
Whenever $\ox$ has a jump at time $t$, all the particles occupying $x$ at time $t$ are killed. 

We give an answer to the following question:  
\begin{center}
For which values of $\lambda,\kappa$ and $q$ is the probability that the branching random walk survives strictly positive?
\end{center}
A priori, the answer might depend on the realization of the random environment, but we will see that the survival probability is either zero, for almost all environments, or strictly positive, for almost all environments.

Let us comment on the dependence on the parameters of the model: It is clear by a coupling argument that increasing $\lambda$ will increase the probability of survival, simply because there are more particles. Similarly, replacing the distribution $q$ of the number of descendants by some distribution $\widetilde q$ having a larger mean should also increase the chance of survival. The dependence on $\kappa$ is more tricky: If the jump rate is small, the process is essentially frozen and remains concentrated on few sites, and can be killed quickly if the environment is particularly unfavorable in a small area. If we increase $\kappa$, the process will jump away from any small area that is atypical and see an environment that is more average. However even in the best case particles will be killed at rate $1$. 

We will not fully resolve the dependence on $\kappa$, but instead connect the problem to the survival rate 
in the one-particle model, which was studied in \cite{shiga}. This correspondence is similar to the connection between the random polymer model and branching random walks in random space-time-environments, as explained in Section 1.3 in \cite{yoshida}. The proof of extinction in the critical case borrows heavily from the proof given in \cite{garetmarchand}, which confirmed Conjecture 1 in \cite{yoshida}.

Branching random walks in time-dependent environments have been studied extensively in the context of the parabolic Anderson model, see \cite{gaertnerdenHoll}, \cite{ErhdenHollMaillard}. However, most papers consider the solution to an SDE with random potential which describes the behavior of the {\bf expectation} of the number of particles in a branching random walk in random environment, and not the actual particle system (a notable exception where the two models are compared, is \cite{ortgieseroberts}).  In addition,
most papers have non-degeneracy conditions on the killing rates which are violated by our environment. In particular, we point out that our model differs from the branching random walks considered in \cite{yoshida} not only because time is continuous instead of discrete, but also because disasters in the environment were excluded in \cite{yoshida} (see formula (1.7)). The possibility of killing many particles at the same site at once makes our model interesting but also creates some technical difficulties.
For a survey on the parabolic Anderson model and random walk in random potential, we refer to \cite{WolfgangK}.

The paper is organized as follows.
In the remainder of Section \ref{sec:intro} we define the process and recall some previously known results about the one-particle model. Our main result, stated in Section \ref{sec:result}, is Theorem \ref{thm:main} which characterizes the set of parameters where the survival probability is strictly positive. 

The subcritical case of Theorem \ref{thm:main} follows immediately from the first moment method, see Section \ref{sec:sub}. 

In Section \ref{sec:super} we handle the supercritical case by comparing our process to an embedded Galton-Watson process with i.i.d.\ offspring distributions. While this argument is relatively short, it needs an auxiliary result (Proposition \ref{prop:propSTilde}) about the one-particle model. To prove the auxiliary result, we need uniform moment bounds (see Proposition \ref{prop:unifMom}) and a concentration inequality (see Proposition \ref{concin}). The proofs of these propositions make use of stochastic domination.
These results can be found in Sections \ref{sec:auxiliary}, \ref{sec:concentration} and \ref{sec:stilde}, in which no branching processes occur.

Finally the critical case follows from a standard comparison to oriented site percolation, presented in Section \ref{sec:outline}. To implement this argument we need two propositions, the proofs of which are carried out in the remainder of Section \ref{sec:critical}.

\subsection{Definition and Notation}\label{sec:constr}


We first define the branching random walk introduced above: We identify the nodes of a tree with the set
\[\N^*\coloneqq \bigcup_{k=0}^\infty \N^k=\Big\{x=(x_1,...,x_k)\colon k\in\N,x_1,...,x_k\in\N\Big\}.\]
We call $|(x_1,...,x_k)|=:k$ the \textbf{height} of $v$ and write $\emptyset$ for the unique element of height $0$, which we call the \textbf{root}. Proceeding recursively we interpret $(x_1,...,x_k)$ as the the $x_k^\text{th}$ child of $(x_1,...,x_{k-1})$, for $k\geq 1$. 
Fix now positive values $\kappa$ and $\lambda$ as well as a distribution $q=(q(k))_{k\in\N}$ on the natural numbers satisfying
\begin{equation}\label{mdef}
m\coloneqq \sum_{k=0}^\infty kq(k) < \infty\quad\text{ and }\quad q(1)<1.
\end{equation}
We associate to every node an exponential clock of rate $\lambda$, and whenever a clock rings the node is removed and replaced by its children, where the number of children is distributed according to $q$. The clocks and the numbers of descendants are independent. We will write $V(t)$ for the set of nodes that are alive at time $t$, starting with $V(0)=\{\emptyset\}$. 

Next, we extend this by associating to each node $v$ alive at time $t$ a position $X(t,v)$ in $\Z^d$. We let each particle perform a simple random walk in continuous time of jump rate $\kappa$ between its birth and the time when it is replaced by its children, independently from everything else. The root initially starts in the origin, and all other nodes start at the position occupied by their parent node at the time of birth. 

For $v\in V(t)$, it will be convenient to extend $X(t,v)$ to a function $X(\cdot,v)\colon [0,t]\to\Z^d$, where for $s\in[0,t]$ we set $X(s,v)$ equal to the position occupied at time $s$ by the unique ancestor of $v$ in $V(s)$.

The process described so far is well-studied. Recall that the environment $\omega=\big(\ox\big)_{x\in\Z^d}$ consists of independent Poisson processes of rate $1$ indexed by the sites of $\Z^d$, which are independent of the random variables defined before. Let
\[\delta(t,x)\coloneqq  \ox(t)-\ox(t^-).\]
If $\delta(t,x) = 1$, we say that
there is a disaster at time $t$ at $x$. The process we are interested in is denoted $(Z(t))_{t\geq 0}$, with
\[Z(t)\coloneqq \big\{v\in V(t)\colon \delta(s,X(s,v))=0\text{ for all } 0\leq s< t\big\}\subseteq V(t).\]
So $Z(t)$ contains all particles $v$ where no disaster occurred along the trajectory of $v$ before time $t$. Note that since we did not assume $q(0)=0$ it is possible that a particle has zero children, and the process may die out even without the influence of the environment.

We will use $Q$ to denote the law of the environment, and $P$ for the law of the branching random walk. Typically we consider the processes $Z(t)$ for fixed realizations of $\omega$, and then we write $P_\omega$ for the \textbf{conditional} or \textbf{quenched} law. The \textbf{annealed} or \textbf{averaged} law $\P$ is given by
\[\P(Z\in\cdot)\coloneqq  \int P_\omega(Z\in\cdot) Q(\dd \omega).\]
We denote the corresponding expectation by $\E$. With a slight abuse of notation, we also use $\E$ for the expectation with respect to $Q$. Occasionally we want to stress the dependence on the parameters, in which case we write $\P^{\kappa,\lambda}$ and $P_\omega^{\kappa,\lambda}$.


\subsection{Previous results about the one-particle model}\label{sec:prev}
There is a close relationship between our model and the model considered in \cite{shiga}. There, the process consists of a single particle performing random walk at rate $\kappa$ among disasters in the same way that particles in our model do. In this section we summarize some known results. 
\par Let $(X(t))_{t\geq 0}$ be a simple random walk in continuous time, moving in $\Z^d$ at a jump rate $\kappa>0$, with the corresponding probability measure denoted $P$. The environment $\omega=\big(\ox\big)_{x\in\Z^d}$ is the same as before. We let $\tau$ be the first time the random walk hits any of the disasters, that is 
\[\tau\coloneqq \inf\Big\{t\geq 0\colon \delta(s,X_s)>0\Big\}\, .\]
We are interested in the probability to survive until time $t$ for a fixed realization of the environment:
\[S(t)\coloneqq P_\omega(\tau \geq t)\]
Note that by averaging over the environments one easily gets the annealed survival rate:
\[\E[S(t)] = \int S(t) \dd Q = e^{-t}.\]
We summarize the results of  \cite{shiga} in the following 
\begin{ttheorem}
Define $p\colon (0,\infty)\to (-\infty, 0)$ by
\begin{equation}\label{pdef}
p(\kappa)\coloneqq \lim_{t\to\infty}\frac 1t\log S(t).
\end{equation}
Then
\begin{itemize}
\item[(i)] The limit in \eqref{pdef} exists $Q$-almost surely and is deterministic, with 
\begin{equation}\label{expectaswell}
p(\kappa)=\lim_{t\to\infty}\frac 1t\E[\log S(t)]\, .
\end{equation}
\item[(ii)] For all $\kappa>0$ we have $p(\kappa)\leq -1$.
\item[(iii)] For any $d$ we have $\lim_{\kappa\to 0}p(\kappa)=-\infty$ and $\lim_{\kappa\to\infty}p(\kappa)=-1$.
\item[(iv)] There exists a critical rate $\kappa_c=\kappa_c(d)\in(0,\infty]$, such that
\[\begin{matrix}p(\kappa)<-1&\text{if }\kappa<\kappa_c\\p(\kappa)=-1&\text{if }\kappa>\kappa_c\end{matrix}\]
\item[(v)] For $d\geq 3$ we have $\kappa_c(d)<\infty$.
\end{itemize}
\end{ttheorem}

\subsection{The main result}\label{sec:result}
We are interested in the event
\begin{equation}
\label{survivaldef}
\{Z\text{ survives}\}\coloneqq \{|Z(t)|>0, \, \forall t\geq 0\}\, .
\end{equation}
Using the exponent $p(\kappa)$ we prove the following criterion:
\begin{theorem}\label{thm:main}
\[{P_\omega}(Z\text{ survives})>0\quad Q\text{-a.s. }\iff \lambda(m-1)+p(\kappa)>0\, .\]
\end{theorem}
In analogy to classical branching processes, we define three regimes.
\begin{definition}
We say that the process $Z(t)$ is 
\[\begin{matrix}
\text{subcritical}&\text{if }&\lambda(m-1)+p(\kappa)<0,\\
\text{critical}&\text{if }&\lambda(m-1)+p(\kappa)=0,\\
\text{supercritical}&\text{if }&\lambda(m-1)+p(\kappa)>0.
\end{matrix}\]
\end{definition}
An easy corollary is
\begin{corollary}\label{annealedsurvival}
\[\P(Z\text{ survives})>0\iff \lambda(m-1)+p(\kappa)>0\, .\]
\end{corollary}
We define the event of \textbf{local survival} to be
\[\{Z\text{ survives locally}\}\coloneqq \{ 0\text{ is occupied for arbitrarily large times}\}\, .\]
Clearly
\[\{Z\text{ survives }\} \supseteq \{Z\text{ survives locally}\}\, .\]
Our proof of Theorem \ref{thm:main} shows in fact that the process survives locally in the supercritical case, so that the following holds.
\begin{corollary}\label{localsurvival}
The process either has a positive probability to survive locally in almost every environment, or it dies out with probability $1$ in almost all environments. Moreover
\[{P_\omega}(Z\text{ survives locally})>0\quad Q\text{-a.s. }\iff \lambda(m-1)+p(\kappa)>0\, .\]
\end{corollary}
\begin{corollary}\label{expgrowth}There exists $c>0$ such that $Q$-almost surely
\[\{Z\text{ survives }\}=\Big\{\liminf_{t\to\infty}|Z_t|e^{-ct}>0\Big\}\]
\end{corollary}
For the proof see Remark \ref{rm:proof}.
\begin{remark}
By an obvious truncation argument, the assumption $m < \infty$ can be 
dropped; if $m = \infty$, we are in the supercritical case.
\end{remark}
We do not make any assumption on the shape of $p$, so a priori it may be discontinuous or may not be increasing in $\kappa$. In Corollary 4.1 in \cite{shiga2} continuity of $p$ is proven for a related class of models, but the relevant case of hard obstacles is excluded. However, if we interpret $p$ as the free energy of a polymer in random environment as in Section 3 of \cite{yoshida_path}, it is reasonable to conjecture that $p$ is concave. A proof might be attempted by showing the following
\begin{conjecture}
Fix a branching mechanism with $m>1$, and set
\[U\coloneqq \Big\{(\kappa,\lambda)\colon \P^{\kappa,\lambda}\big(Z\text{ survives}\big)>0\Big\}\subseteq (0,\infty)^2.\]
Then $U$ is a convex set.
\end{conjecture}

\subsection{Some more notation}\label{sec:notation}

Before we start with the proof of Theorem \ref{thm:main}, we collect some notation that will be useful at various points throughout this work. We first extend the definition of $Z$ to the case where we may have more than one initial particle. 

We call $\eta=(\eta_x)_{x\in\Z^d}$ a \textbf{configuration}, and let $Z^\eta$ denote the process as defined before, except that we start with $\eta_x$ particles in $x$, all of which evolve independently but in the same environment. If $A\subseteq \Z^d$ and $R\geq 0$ is an integer we record the special configuration $(A,R)$ where each site $x\in A$ is occupied by $R$ particles, that is
\begin{align}\label{eq:special}
(A,R)_x:=R\,\1_A(x)\quad\text{ for all }x\in\Z^d.
\end{align}
For $A\subseteq \Z^d$ we use $Z^A$ instead of $Z^{(A,1)}$ for the process started from exactly one particle on every site in $A$. For $t>0$ and $\eta$ a configuration we use 
\begin{align}\label{eq:definitialconf}
Z^{\{t\}\times \eta}=\big(Z^{\{t\}\times \eta}(s)\big)_{s\geq t}
\end{align}
to denote the process started at time $t$ with $\eta_x$ particles occupying each site $x$, and we use $Z^{\{t\}\times A}$ if $\eta$ is equal to $(A,1)$.

Moreover if $(Z(t))_{t\geq 0}$ is some branching process and $B\subseteq\Z^d$, we let $(Z_{B}(t))_{t\geq 0}$ denote the \textbf{truncated process} consisting of all particles that have never left $B$:
\begin{equation}
\label{trunc}
Z_{B}(t)\coloneqq \Big\{v\in Z(t)\colon X(s,v)\in B\;\text{ for all }s\in[0,t]\Big\}.
\end{equation}
In the simple case where $B=\{-L,...,L\}^d$ for some $L\in\N$ we simply write $(Z_L(t))_{t\geq 0}$. We also use the following notation for the set of particles of $(Z_t)_{t\geq 0}$ occupying a site $x$ at time $t$:
\begin{equation}\label{particlesat0}
Z(t)\cap\{x\}\coloneqq \{v\in Z(t)\colon X(t,v)=x\}.
\end{equation}
If $\eta$ is a configuration, we denote the event that at time $t$ every site is occupied by at least $\eta_x$ particles by
\begin{align}\label{substeil}
\{\eta\leq Z(t)\}\coloneqq \Big\{\eta_x\leq \big|Z(t)\cap\{x\}\big|\text{ for all }x\in\Z^d\Big\}.
\end{align}
In the case where $\eta=\1_C$ for some $C\subseteq\Z^d$ this is simply written as 
\begin{align*}
\{C\subseteq Z(t)\}\coloneqq \{(C,1)\leq Z(t)\}=\Big\{\forall x\in C\;\exists v\in Z(t)\text{ such that }X(t,v)=x\Big\}.
\end{align*}


\section{The subcritical case}\label{sec:sub}

\begin{proof}[Proof of Theorem \ref{thm:main} (subcritical case)]
Assume that
\[-\eps\coloneqq \lambda(m-1)+p(\kappa)<0.\]
For almost all environments $\omega$, we can find $T = T(\omega)$ such that 
\[S(t)=P_\omega(\tau\geq t)\leq e^{t(p(\kappa)+\frac\eps 2)}\quad \forall t\geq T.\]
Then we have for $t \geq T(\omega)$
\[E_\omega[|Z(t)|]=E_\omega\Big[\sum_{v\in V(t)}\1_{\{v\text{ survives until }t\}}\Big]\]
\begin{equation}\label{eq:comp1}
=E[|V(t)|]S(t)=E[m^M] S(t)=e^{\lambda(m-1)t}S(t)\leq e^{-\frac{\eps}{2} t},
\end{equation}
where $M$ is a random variable whose law is
Poisson with parameter $\lambda t$. This implies $Z(t)\to 0$ for almost all environments.
\end{proof}

\section{The supercritical case}\label{sec:super}

For the proof in the supercritical case we will need to consider the random variable
\begin{equation}\label{stilddef}
\widetilde S(t)\coloneqq  P_\omega(\tau\geq t,X_t=0).
\end{equation}
It is intuitively clear that $\widetilde S(t)$ should decay to zero with the same exponential rate as $S(t)$, since the event $\{X(t)=0\}$ has probability decaying only with a polynomial rate, and therefore its contribution should be dominated by the contribution of the event $\{\tau\geq t\}$. This is stated in the following 

\begin{proposition}\label{prop:propSTilde}
It holds that
\begin{equation}\label{eq:decayat0}
\lim_{t\to\infty} \frac{1}{t}\E[\log \widetilde S(t)]=p(\kappa)
\end{equation}
and
\begin{equation}\label{eq:almostsuredecayat0}
\lim_{t\to\infty} \frac{1}{t}\log \widetilde S(t)=p(\kappa)\quad\text{ for $Q$-almost all }\omega.
\end{equation}
Moreover for any $t\geq 0$ we have
\begin{equation}\label{eq:finirate}
 \E[\log \widetilde S(t)]>-\infty.
\end{equation}
\end{proposition}
\begin{proof}[Proof of Theorem \ref{thm:main} (supercritical case)]
Assume
\begin{equation}\label{eq:positive}\lambda(m-1)+p(\kappa)>0\, .\end{equation}
We will find a branching process with i.i.d.\ offspring distributions embedded in $Z$. More precisely, we introduce a process $(A(k))_{k\in \N}$ taking values in $\N$, such that we have $A(k)\leq |Z(k T)|$ for all $k\in\N$ and some $T>0$. The claim then follows by showing that in almost all environments the event $\{A(k)>0$ infinitely often$\}$ has positive probability.

Fix some large $T$, and set $A(0)\coloneqq 1=|Z_0|$ and 
\[A(k)\coloneqq \Big|\Big\{v\in Z(kT)\colon X(iT,v)=0\text{ for all }i=0,...,k\Big\}\Big|\]
That is, for the process $A$ we only consider particles that return to the origin at times $T,2T,3T,...\,$. Note that every particle that contributes to $A(k)$ is the descendant of a particle that contributed to $A(k-1)$. 

To see that $(A(k))_k$ has i.i.d.\ offspring distributions, we recall from Section \ref{sec:notation} the notation $Z(t)\cap\{0\}$ and $Z^{t,A}$. Using those, we can define the sequence $(q^{(k)})_{k\in\N}$ of offspring distributions by
\[q^{(k)}(j)=P_\omega\big(\big|Z^{(k-1)T,\{0\}}({kT})\cap \{0\}\big|=j\big)\quad \text{for }j\in\N.\]
Note that $q^{(k)}$ only depends on the environment in the interval $[(k-1)T,kT)$, and $(q^{(k)})_k$ is therefore an i.i.d.\ sequence in the space of probability measures on $\N$. 

We let $m^{(k)}$ denote the expectation of $q^{(k)}$. By a well-known result on branching processes with i.i.d.\ offspring distributions, see \cite{smith, tanny}, the survival probability of 
$(A(k))_{k\in\N}$ is positive for almost all environments if
\begin{equation}\label{eq:cond1}\E[\log(1-q^{(1)}(0))]>-\infty\end{equation}
and
\begin{equation}\label{eq:cond2}\E[\log(m^{(1)})]>0\, .\end{equation}
We can write $m^{(1)}$ as 
\[m^{(1)}=\sum_{j\in \N}j q^{(1)}(j)=\sum_jjP_\omega\big(\big|Z({T})\cap \{0\}\big|=j\big)=E_\omega\big[|Z(T)\cap\{0\}|\big]\, .\]
Recall the definition of $\widetilde S(t)$ in \eqref{stilddef} By the same computation as in (\ref{eq:comp1}) we get
\begin{equation}\label{mrandom}
m^{(1)}=e^{\lambda(m-1)T}\widetilde S(T).
\end{equation}
In order to give a lower bound for the quantity in (\ref{eq:cond1}), we compare the branching process to the random walk of a single particle: We choose a path by starting in the root, and whenever there is more than one descendant, we follow its first child. Let $F(t)$ be the event that this construction succeeds up to time $t$, that is the currently observed particle always has at least one descendant. We have
\[P^{\kappa,\lambda}(F(t))=E[(1-q(0))^M]=\exp(-\lambda tq(0))\]
where $M$ is the number of branching events along this path. Note that $M$ has distribution Poisson($\lambda t$), so that $1-q^{(1)}(0)\geq E[F(T)]\widetilde S(T)$. By \eqref{eq:finirate} in Proposition \ref{prop:propSTilde} we see that indeed
\[ \E\big[\log(1-q^{(1)}(0))\big]\geq -\lambda Tq(0)+\E[\log \widetilde S(T)]>-\infty.\]

We can now conclude: By \eqref{eq:decayat0} and \eqref{mrandom}, we find for every $\eps>0$ some $T$ large enough that
\[ \E[\log(m^{(1)})]\geq T\big(\lambda(m-1)+(p(\kappa)-\eps)\big).\]
By (\ref{eq:positive}), we can satisfy (\ref{eq:cond2}) by choosing $\eps$ small enough, finishing the proof.
\end{proof}
\begin{remark}\label{rm:proof}
The proof shows in fact that in the supercritical case, the process survives locally with positive probability. Using results of \cite{tanny} about branching processes with i.i.d.\ offspring distributions we also see that in the supercritical case, the number of particles grows exponentially fast.
\end{remark}

It remains to prove Proposition \ref{prop:propSTilde}. This will take up most of Section \ref{sec:super}: We start by proving a uniform moment bound in Section \ref{sec:auxiliary} using comparison techniques from \cite{shiga} and some results about stochastic orders. In Section \ref{sec:concentration} we use this to get a concentration inequality, which is necessary for the proof of Proposition \ref{prop:propSTilde} in Section \ref{sec:stilde}. 


\subsection{A uniform moment bound}\label{sec:auxiliary}
The following proposition is key to proving the concentration inequality in the next section:

\begin{proposition}\label{prop:unifMom}
For every $\delta\in(0,1)$ there is some $C>0$ such that 
\[\sup_{x\in\Z^d}\E\Big[P_\omega(\tau\geq 1|X(1)=x)^{-\delta}\Big]<C<\infty.\]
\end{proposition}

For the proof we use an equivalence relation $\equiv$ on $\Z^d$ defined by
\[(y_1,...,y_d)\equiv (z_1,...,z_d)\quad\iff\quad y_1= z_1\mod 2.\]
We will identify $\Z^d/_\equiv$ with $\Z_2=\{0,1\}$, and we use $\pi\colon \Z^d\to\{0,1\}$ to denote the projection. Let $\widetilde \omega$ be an environment on $\{0,1\}$, consisting as usual of two independent Poisson processes $\widetilde  \omega^{(0)}$ and $\widetilde \omega^{(1)}$ of rate $1$. We write $\lobar$ for the environment on $\Z^d$ given by
\[(\lobar)^{(y)}=\omega^{(\pi(y))}\quad\text{ for }y\in\Z^d.\]
Note that this is a degenerate environment on $\Z^d$, where all sites that share an equivalence class in $\equiv$ experience the same disasters. We will slightly abuse notation by writing $\E$ for the law of $\widetilde \omega$ as well.

First we need the following auxiliary lemma.
\begin{lemma}\label{lem:expmom}
Let $(\widetilde X(t))_{t\in[0,1]}$ be simple random walk on $\{0,1\}$ of jump rate $\kappa$. Then for any $p\in(0,1)$ we have
\begin{align}\label{eq:expmom1}\sup_{i=0,1}\E\big[P_{\widetilde \omega}^\kappa\big(\tau\geq 1,\widetilde X(1)=i\big)^{-p}\big]<\infty.\end{align}
while for $p\in(0,\frac 12)$ we have
\begin{align}\label{eq:expmom2}\sup_{t\in[0,1]}\E\big[P_{\widetilde \omega}^\kappa\big(\tau\geq t,\widetilde X(t)=0\big)^{-p}\big]<\infty.\end{align}
\end{lemma}
\begin{proof}
This is a modification of the proof of Lemma 2.4 in \cite{shiga}, where the integrability of $P_{\widetilde \omega}(\tau\geq 1)^{-p}$ is shown. We quickly sketch how the proof can be modified: 

Note that the bound in (2.24) of \cite{shiga} is actually a bound for $P_\omega(\tau\geq t,X(t)=1)$. By slightly modifying the argument we obtain a similar bound for $P_\omega(\tau\geq t,X(t)=0)$, where on the right hand side we have to replace $C(t)C_1(t)^nt_1^{1-2p}$ by $C(t)^2C_1(t)^{n-1}t_1^{-p}$. It is clear that this does not make a difference for the convergence of the sum appearing in the display after (2.26), where the coefficients $\beta_n$ have to be replaced by
\begin{align*}
\beta_n':=\frac{\Gamma(1-p)^2\Gamma(2-2p)^{n-1}}{\Gamma(2n(1-p))}t^{(2-2p)n-1}.
\end{align*}
This gives the first claim, and the second claim follows because the coefficients in that sum can be chosen increasing in $t$. This is clear for $C(t)$ and $C_1(t)$, and for $n\geq 1$ and $p\in(0,\frac 12)$ the $\beta_n'$ are increasing as well. 
\end{proof}

\begin{proof}[Proof of Proposition \ref{prop:unifMom}]
By Lemma 2.2 in \cite{shiga} we have
\[\E\big[(P^{\kappa}_{\omega}(\tau\geq 1,X(1)=x))^{-\delta}\big]\leq \E\big[(P^{\kappa}_{\pi^{-1}(\widetilde \omega)}(\tau\geq 1,X(1)=x))^{-\delta}\big]\]
and dividing both sides by $\big(P^\kappa(X(1)=x)\big)^{-\delta}$ gives
\[\E\big[(P^{\kappa}_{\omega}(\tau\geq 1|X(1)=x))^{-\delta}\big]\leq \E\big[(P^{\kappa}_{\pi^{-1}(\widetilde \omega)}(\tau\geq 1|X(1)=x))^{-\delta}\big]\]
Moreover by \eqref{eq:expmom1} we have
\[\sup_{x\in\Z^d}\E\Big[\big(P^{\frac\kappa 2}_{\pi^{-1}(\widetilde \omega)}(\tau\geq 1|X(1)\equiv x)\big)^{-\delta}\Big]<\infty.\]
So the claim follows once we show that 
\begin{equation}\label{eq:mitte}
\E\Big[\big(P^{\kappa}_{\pi^{-1}(\widetilde \omega)}(\tau\geq 1|X(1)=x)\big)^{-\delta}\Big]\leq \E\Big[\big(P^{\frac\kappa 2}_{\pi^{-1}(\widetilde \omega)}(\tau\geq 1|X(1)\equiv x)\big)^{-\delta}\Big]
\end{equation}

For simplicity we only treat the case where $x\equiv 0$, noting that the case $x\equiv (1,0,...,0)$ is similar. For a fixed environment $\widetilde \omega$, let $N$ be the number of disasters in $[0,1]$. We write $T_1,...,T_N$ for the disaster times in increasing order, and $E_1,...,E_N$ for their locations. Let us write $\P^{T_1,...,T_N}$ resp. $\E^{T_1,...,T_N}$ for the law resp. expectation of $E_1,...,E_N$ conditioned on $N$ and $T_1,...,T_N$, which is simply the uniform distribution on $\{0,1\}^N$. Notice that for any event $A$ and function $f:(0,1]\to\R$ we can write
\begin{align*}
\E^{T_1,...,T_N}\big[f(P_{\pi^{-1}(\widetilde\omega)}^\kappa(\tau\geq 1|A))\big]=\frac{1}{2^N}\sum_{(e_1,...,e_N)\in\{0,1\}^N} f(\alpha(e_1,...,e_N))
\end{align*}
where $\alpha$ is a measure on $\{0,1\}^N$ defined by
\begin{align*}
\alpha(e_1,...,e_N)\coloneqq P^\kappa\big(\pi(X(T_1))=1-e_1,...,\pi(X(T_N))=1-e_N\big|A\big).
\end{align*}

Before we make use of this observation we introduce a different encoding for the disaster locations which will be convenient later: Given a realization $\widetilde\omega$, we define its configuration $I_{\widetilde\omega}=\big(I_{\widetilde\omega}(i)\big)_{i=0}^N$ by
\begin{align*}
I_{\widetilde\omega}=\Big(\1\{E_1=0\},\1\{E_{2}\neq E_1\},...,\1\{E_N\neq E_{N-1}\},\1\{E_N=0\}\Big)\in\{0,1\}^{N+1}.
\end{align*}
The intuition is that $I_{\widetilde \omega}$ encodes the necessary jumps for the random walk. To see this we notice that if $\{I_{\widetilde\omega}(i)=1\}$ for some $i\in\{0,...,N\}$, the process has to switch sites in $[T_i,T_{i+1})$ if it wants to survive and at time $1$ end up in a location equivalent to $0$. (Recall that $T_0= 0$ and $T_{N+1} = 1$).

We let $\Sigma\subseteq\{0,1\}^{N+1}$ be the set of configurations with an even number of $1$s. Observe that $\P^{T_1,...,T_N}(I_{\widetilde\omega}\in\Sigma)=1$, and that $I_{\widetilde\omega}$ has the uniform distribution on $\Sigma$. 

On the other hand, we define for a \cadlag process $X$ on $\Z^d$ its signature $\I$ as
\[\I:=\Big(\pi\big(X(T_1)\big),\pi\big(X(T_2)-X(T_1)\big),...,\pi\big(X(T_N)-X(T_{N-1})\big),\pi\big(X(1)-X(T_N)\big)\Big).\]
We notice that $\{\I\in\Sigma\}=\{X(1)\equiv 0\}$ and $\{\I=I_{\widetilde\omega}\}=\{\tau\geq 1,X(1)\equiv 0\}$, so that we can introduce two probability measures $\mu$ and $\nu$ on $\Sigma$ by setting
\begin{equation}\label{eq:munudef}
\mu(I):=P^{\kappa}\big(\I=I\big|X(1)=x\big)\quad\text{ and }\quad\nu(I):=P^{\frac\kappa 2}\big(\I=I\big|X(1)\equiv x\big).
\end{equation}
Notice that we now have
\[P_{\pi^{-1}(\widetilde\omega)}^\kappa(\tau\geq 1|X(1)=x)=\mu(I_{\widetilde \omega})\quad\text{ and }\quad P_{\pi^{-1}(\widetilde\omega)}^{\frac\kappa 2}(\tau\geq 1|X(1)\equiv x)=\nu(I_{\widetilde \omega}).\]

So we have two probability measures which are evaluated at a random point $I_{\widetilde\omega}$, and we want to compare the expectations of $f(\mu(I_{\widetilde\omega}))$ and $f(\nu(I_{\widetilde\omega}))$ for the convex function $f(x)=x^{-\delta}$. 

For this we recall some results about stochastic orders: For two probability measures $\mu$ and $\nu$ on $\Sigma$, we say that $\mu$ is \textbf{majorized} by $\nu$, denoted $\mu\preceq_M\nu$, if 
\[\sum_{i=1}^k\mu(a_i)\leq\sum_{i=1}^k\nu(b_i)\quad\text{ for all }k=1,...,2^N,\]
where $\Sigma=\{a_1,...,a_{2^N}\}=\{b_1,...,b_{2^N}\}$, and the ordering is chosen in such a way that
\[\mu(a_1)\geq ...\geq \mu(a_{2^N})\quad\text{ and }\quad\nu(b_1)\geq...\geq \nu(b_{2^N}).\]

The intuition for $\mu\preceq_M\nu$ is that the mass of $\mu$ is more spread out than the mass of $\nu$, so that the random evaluation $\mu(I_{\widetilde\omega})$ should be more random than $\nu(I_{\widetilde\omega})$. The following result makes this precise in terms of the convex stochastic order:
\begin{lemma}[Corollary 1.5.37 in \cite{mullerstoyan}]\label{lem:corollary}We have $\mu\preceq_M\nu$ if and only if 
\begin{align*}
\frac 1{|\Sigma|}\sum_{\sigma\in\Sigma}f(\nu(\sigma))\leq \frac 1{|\Sigma|}\sum_{\sigma\in\Sigma}f(\mu(\sigma))\quad\text{for all convex functions }f:(0,1]\to\R.
\end{align*}
\end{lemma}
Indeed we have
\begin{lemma}\label{lem:majo}
Let $\mu$ and $\nu$ be defined as in \eqref{eq:munudef}. Then $\mu\preceq_M\nu$.
\end{lemma}
Hence Lemma \ref{lem:corollary} implies 
\[\E^{T_1,...,T_N}\big[f\big(\mu(I_{\widetilde\omega})\big)\big]\leq \E^{T_1,...,T_N}\big[f\big(\nu(I_{\widetilde\omega})\big)\big]\]
for all convex functions $f:(0,1]\to\R$. Inserting $f:x\mapsto x^{-\delta}$, this in particular shows \eqref{eq:mitte} by taking expectations.
\end{proof}

It remains to show Lemma \ref{lem:majo}. If $Z$ is a \cadlag process, we call $t$ a jump time of $Z$ if $Z(t)\not\equiv Z(t^-)$, and we write $R_Z$ for the number of jumps times of $Z$ in $[0,1]$.

\begin{lemma}\label{lem:numjumps}
Let $X$ resp. $Y$ be simple random walks on $\Z^d$ with jump rate $\kappa$ resp. $\frac\kappa 2$. Then 
\[R_Y\big|\{Y(1)\equiv x\}\quad \preceq_{st}\quad R_X\big|\{X(1)=x\}\]
where $\preceq_{st}$ denotes stochastic domination.
\end{lemma}
\begin{proof}
It is easier to show that 
\[R_Y\big|\{Y(1)\equiv x\}\quad \preceq_{lr}\quad R_X\big|\{X(1)=x\}\]
where $\preceq_{lr}$ denotes domination in the likelihood ratio order, see for example Chapter 1.4 in \cite{mullerstoyan}, where it is also shown that $\preceq_{lr}$ is stronger than $\preceq_{st}$. 

We have to check that for $k,l\in\N$ of the same parity as $x_1$ and such that $|x_1|\leq k\leq l$, the following holds:
\[P^{\kappa}(R_X=k|X(1)=x)P^{\frac\kappa 2}(R_Y=l|Y(1)\equiv x)\leq P^\kappa(R_X=l|X(1)=x)P^{\frac\kappa 2}(R_Y=k|Y(1)\equiv x)\]

We apply the definition of conditional probability and cancel the terms that appear on both sides (note that
$P^{\frac\kappa 2}(Y(1)\equiv x | R_Y=l) =1$ since $l$ has the same parity as $x_1$). Then we can rewrite the equation as
\[\frac{P(Z_k=x_1)}{P(Z_l=x_1)}\leq \frac{P(A=l)P(A'=k)}{P(A=k)P(A'=l)}=2^{l-k}.\]
Here $(Z_i)_{i\in\N}$ is a discrete time simple random walk on $\Z$, and $A$ resp. $A'$ is a Poisson random variable of parameter $\frac{\kappa}d$ resp. $\frac\kappa {2d}$. But this inequality holds, since by the Markov property
\[P(Z_l=x_1)\geq P(Z_k=x_1)P(Z_{l-k}=0)\geq P(Z_k=x_1)2^{-(l-k)}.\]
\end{proof}

Now we are ready to show Lemma \ref{lem:majo}.
\begin{proof}
Let us define weights $p_0,...,p_N$ by $p_0:=T_1$, $p_N:=1-T_N$ and $p_i:=T_{i+1}-T_i$ for all other values of $i$. We note that $\mu$ and $\nu$ do not depend on the order of $T_1,...,T_N$, and therefore we can rearrange them to satisfy
\begin{equation}\label{eq:ordered}
p_0\leq p_1\leq ...\leq p_N.
\end{equation}
Now for $k\in\N$, let $M_k=(M_k(0),...,M_k(N))$ denote a random variable having the multinomial distribution with $k$ trials, and write $P_k$ for its law. That is, $k$ indistinguishable balls are thrown in bins numbered $0,...,N$ such that each ball independently lands in bin $i$ with probability $p_i$, and $M_k(i)$ is the final number of balls in bin $i$. 
We define 
\begin{equation}\label{eq:defIk}\I_k:=\big(\1\{M_k(0)\text{ is odd}\},...,\1\{M_k(N)\text{ is odd}\}\big)\in\{0,1\}^{N+1}.\end{equation}
We will often use $\I_k$ interchangeably with the set $\{i:M_k(i)\text{ is odd}\}\subseteq \n$, where $\n=\{0,...,N\}$. Consider random variables $K$ and $L$ taking values in $\N$ with
\[P(L=l)=P^{\kappa}(R_X=l|X(1)=x)\quad\text{ and }\quad P(K=k)=P^{\frac \kappa 2}(R_Y=k|Y(1)\equiv x).\]
Observe that by conditioning on $R_X$ and $R_Y$ we get
\begin{equation}\label{eq:mixture}\mu(I)=E[P_{L}(\I_L=I)]\quad\text{ and }\quad \nu(I)=E[P_{K}(\I_K=I)]\end{equation}

Indeed, conditional on a random walk having $K$ jumps in $[0,1]$, each jumps occurs in $[T_{i},T_{i+1})$ with probability $p_i$, independently of the other jump times, and the process switches sites between $T_i$ and $T_{i+1}$ exactly if there is an odd number of jumps in $[T_i,T_{i+1})$.

One might be tempted to think that we are done now, since for all fixed values $k\leq l$ we can easily show that $P_l(\I_l\in\cdot)\preceq_M P_k(\I_k\in\cdot)$ holds: The distribution of $\I_l$ can be obtained from the distribution of $\I_k$ by the application of a doubly stochastic matrix, and this is an equivalent characterization of $\preceq_M$, see for example Theorem 1.5.34 in \cite{mullerstoyan}. Moreover from Lemma \ref{lem:numjumps} we know that there exists a coupling between $K$ and $L$ such that $K\leq L$ holds with probability one. However the majorization order is not stable under taking mixtures, so this does not give the conclusion.

Instead we define a partial order $\preceq$ on $\Sigma$ by 
\[(i_0,...,i_N)\preceq (j_0,...,j_N)\quad\iff\quad\sum_{l=0}^ki_l\leq \sum_{l=0}^k j_l\quad\text{ for all }k=0,...,N.\]

We will show in Lemma \ref{lem:final} that if we increase the number of jumps from $2k$ to $2k+2$, the mass in $\Sigma$ will become less concentrated on the ''small values`` with respect to this partial order, which is what we need to conclude: 

First note that both $\mu$ and $\nu$ are decreasing in $\preceq$, as defined in part (i) of Lemma \ref{lem:final} below. From \eqref{eq:mixture} we see that this follows by taking expectations in \eqref{eq:increasing}, and noting that both $K$ and $L$ are supported on the even numbers. Moreover we have 
\[\mu(A)\leq \nu(A)\quad\text{ for all decreasing sets }A.\]
To see this, recall from Lemma \ref{lem:numjumps} that we can couple $K$ and $L$ such that $K\leq L$ holds with probability one, and apply \eqref{eq:mixture} together with \eqref{eq:increasingSet}. We have checked conditions (1) and (2) from \cite{majorization}, and $\mu\preceq_M\nu$ now follows from Theorem 3 in that work.
\end{proof}
It remains to show
\begin{lemma}\label{lem:final}
\begin{itemize}
 \item[(i)]$P_{2k}(\I_{2k}\in\cdot)$ is decreasing in $\preceq$. That is, for all $I,J\in\Sigma$ we have
 \begin{equation}\label{eq:increasing}I\preceq J\quad \implies \quad P_{2k}(\I_{2k}=I)\geq P_{2k}(\I_{2k}=J).\end{equation}
 \item[(ii)]Let $A\subseteq \mathcal P(\Sigma)$ be a decreasing set, i.e. $J\in A$ implies $I\in A$ for all $I$ with $I\preceq J$. Then
 \begin{equation}\label{eq:increasingSet}
 P_{2k+2}(\I_{2k+2}\in A)\leq P_{2k}(\I_{2k}\in A).
 \end{equation}
\end{itemize}
\end{lemma}
\begin{proof}
For $S\subseteq \n$ we write $M_k(S):=\sum_{i\in S}M_k(i)$. We recall the following fact about a binomial random variable $B_{n,p}$ with $n$ trials and success probability $p$:
\begin{equation}\label{eq:factbinom}P(B_{n,p}\text{ is even})=\frac 12\big(1+(1-2p)^n\big).\end{equation}
\textbf{Part (i)}: For $S,T\subseteq\n$ disjoint, we consider the function
\[f^S_T(r):=P\big(M_k(i)\text{ is even }\forall i\in S,M_k(j)\text{ is odd }\forall j\in T\big|M_k(S\cup T)=r\big).\]
Whenever $S$ or $T$ is the empty set, we drop it from the notation and just write $f_T$ or $f^S$. We first show \eqref{eq:increasing} in two special cases:

Assume that $I\subseteq J$, with $J\setminus I=:\{a_1,...,a_{2m}\}$. Let $A$ be the event $A:=\{I\subseteq \I_{2k}\subseteq J\}$ and set $S_j:=\{a_{2j-1},a_{2j}\}$. Then
\[P_{2k}(\I_{2k}=I)=P_{2k}(A)E\Big[\prod_{j=1}^m f^{S_j}(M_{2k}(S_j))\Big|A\Big].\]
Clearly $f^{S_j}(m)$ is only positive if $m$ is even, and in this case $f^{S_j}(m)\geq f_{S_j}(m)$ follows from \eqref{eq:factbinom}. But this means 
\[P_{2k}(\I_{2k}=J)=P_{2k}(A)E\Big[\prod_{i=1}^m f_{S_j}(M_{2k}(S_j))\Big|A\Big]\leq P_{2k}(\I_{2k}=I).\]
Next we assume that $|I|=|J|$ and that $I$ and $J$ only differ in two coordinates, that is $I=I_0\cup\{a\}$ and $J=I_0\cup\{b\}$ for some $b<a$. Let $B$ be the event $B:=\big\{I_0\subseteq \I_{2k}\subseteq I_0\cup\{a,b\}\big\}$. Then
\begin{align*}P_{2k}(\I_{2k}=I)=&P(B)E\Big[f_{a}^{b}\big(M_{2k}(\{a,b\})\big)\Big|B\Big]\\
\geq &P(B)E\Big[f_{b}^{a}\big(M_{2k}(\{a,b\})\big)\Big|B\Big]=P_{2k}(\I_{2k}=J).
\end{align*}
For the inequality we have used that for $m$ odd we have
\[f_{a}^{b}(m)=P(B_{m,p}\text{ is even})\geq P(B_{m,p}\text{ is odd})=f_{b}^{a}(m)\]
where $p=\frac{p_{b}}{p_{a}+p_{b}}$. Note that $a>b$ and \eqref{eq:ordered} imply $p\leq \frac 12$.

Now the general case follows from the observation that for any $I\preceq J$ we can find $I_0\preceq ... \preceq I_r$ such that $I_0=I$ and $I_r\subseteq J$, and with the property that $I_{i+1}$ and $I_{i}$ only differ in two coordinates, as defined above.

\textbf{Part (ii)}: We do this by constructing a coupling $(\I_{2k},\I_{2k+2})$ with the property that $\I_{2k}\preceq \I_{2k+2}$ holds with probability one, from which \eqref{eq:increasingSet} follows. 

For this, let $(B,A)$ be chosen from $\{(b,a)\colon 0\leq b\leq a\leq N\}$ according to
\[P\big((B,A)=(b,a)\big)=2p_ap_b\1\{b<a\}+p_a^2\1\{a=b\}\]
and let $M$ be independently sampled with the multinomial distribution with $2k$ trials. On the event $\{A=B\}$ we define $\I_{2k}$ from $M$ according to the definition, and set $\I_{2k+2}$ equal to $\I_{2k}$. 

In the case where $B<A$, we first fix the coupling on $\n\setminus\{A,B\}$ by
\[\I_l(i):=\1\{M(i)\text{ is odd}\}\quad\text{ for }i\notin\{A,B\}\text{ and }l\in\{2k,2k+2\}.\]
Then we set $R:=M(A)+M(B)$ and $p:=\frac{p_B}{p_A+p_B}$, and consider an independent random variable $U$ distributed uniformly in $[0,1]$. From this we define
\begin{align*}
\I_{2k}(B)&:=\1\{U\leq P(B_{R,p}\text{ is odd})\} \\
\I_{2k+2}(B)&:=\1\{U\leq P(B_{R,p}\text{ is even})\}
\end{align*}
Finally, for $l$ equal to $2k$ or $2k+2$, we set 
\begin{equation}\label{eq:parity}\I_l(A):=R-\I_l(B)\mod(2).\end{equation}

We claim that this is indeed the desired coupling. First note that we can sample a realization of the multinomial distribution $M_{2k+2}$ with $2k+2$ trials by sampling $M$ together with two additional balls $A$ and $B$ as described above. If the extra balls end up in the same bin, then the parity of all coordinates of $M$ and $M_{2k+2}$ will agree, and we can take $\I_{2k}=\I_{2k+2}$. 

Otherwise adding $A$ and $B$ will flip the parity of $M(A)$ and $M(B)$. So conditionally on $\{M(A)+M(B)=R\}$ we have sampled $\I_{2k}(B)$ and $\I_{2k+2}(B)$ with the correct laws, which then forces us to choose $\I_{2k}(A)$ and $\I_{2k+2}(A)$ as in \eqref{eq:parity}.

But now \eqref{eq:ordered} and $B<A$ imply $p=\frac{p_B}{p_A+p_B}\leq \frac 12$, so from \eqref{eq:factbinom} we obtain 
\[P(B_{R,p}\text{ is odd})=\frac 12-\frac 12(1-2p)^R\leq \frac 12\leq P(B_{R,p}\text{ is even}).\]
Therefore $\I_{2k}(B)\leq \I_{2k+2}(B)$, which implies $\I_{2k}\preceq \I_{2k+2}$. 

More precisely, if $R$ is even we have 
\[(\I_{2k}(B),\I_{2k}(A))=(1,1)\implies (\I_{2k+2}(B),\I_{2k+2}(A))=(1,1)\]
so that $\I_{2k}\subseteq \I_{2k+2}$ with probability one. If $R$ is odd, we note that 
\[(\I_{2k}(B),\I_{2k}(A))=(1,0)\quad\implies \quad(\I_{2k+2}(B),\I_{2k+2}(A))=(1,0).\]
On the other hand, on $\{(\I_{2k}(B),\I_{2k}(A))=(0,1)\}$ we can have either $\I_{2k+2}=\I_{2k}$ or $(\I_{2k+2}(B), \I_{2k+2}(A)) = (1,0)$. In the second case we have strict inequality, $\I_{2k}\prec \I_{2k+2}$.
\end{proof}

\subsection{A concentration inequality}\label{sec:concentration}

We write
\[S(t,x)\coloneqq P_\omega(\tau\geq t,X(t)=x).\]
With the previous moment bound at hand, we can now proceed to prove a concentration inequality for the sequences $\big(S(t,x)\big)_{t\geq 0}$ where the bounds do not depend on $x$. We follow the proof of Proposition 3.2.1 in \cite{yoshida}.
\begin{proposition}\label{concin}
There exist $c>0$ and $C>0$, such that $\eps\in(0,c)$ implies 
\begin{equation}\label{eq:conc}Q\Big(\big|\log S(t,x)-\E[\log S(t,x)]\big|>\eps t\Big)\leq 2\exp(-C\eps^2t)\end{equation}
for all $t\in\N$ and $x\in\Z^d$.
\end{proposition}
\begin{proof}
We will drop the dependence on $t$ and $x$ in the notation, and only write $S$ for $S(t,x)$. Let $\omega_i$ be the environment that contains all disasters $(t,y)$ of $\omega$ except for those with $t\in[i-1,i)$. We now consider the filtration $\big(\mathcal F_i\big)_{i=0}^t$ with
\[\mathcal F_i\coloneqq \sigma\big(\oy(s)\colon s<i,y\in\Z^d)\]
and the random variables $\big(S_i\big)_{i=1}^t$ given by
\[S_i\coloneqq P_{\omega_i}(\tau\geq t, X(t)=x).\]
Notice that $\E[\log S_i|\mathcal F_i]=\E[\log S_i|\mathcal F_{i-1}]$. Now by Lemma A.1 in \cite{yoshida}, we obtain
\[Q\big(\big|\log S-\E[\log S]\big|>\eps t\big)\leq 2\exp(-C\eps^2t)\]
for some explicit constant $C>0$ once we have shown that
\[\E\Big[e^{\delta|\log S-\log S_i|}\Big|\mathcal F_{i-1}\Big]\leq A\]
holds for some $\delta >0$ and for some $A>0$ not depending on $i$, $t$ or $x$. We have $S\leq  S_i$, and therefore 
\begin{equation}\label{alphaundeta}
e^{\delta|\log S-\log S_i|} = \Big(\frac{S}{S_i}\Big)^{-\delta}=\Big(\sum_{y,z}\alpha_{y,z}\eta_{y,z}\Big)^{-\delta}
\end{equation}
where 
\[\alpha_{y,z}\coloneqq P_{\omega_i}\big(X(i-1)=y,X(i)=z\big|\tau\geq t, X(t) = x\big)\]
and
\[\eta_{y,z}\coloneqq P_\omega^{(i-1,y),(i,z)}\big(\tau \geq i\big).\]
Here $P^{(r,y),(s,z)}$ is the law of a random walk starting at time $r$ in $y$ and conditioned to end up in $z$ at time $s$.
To see that \eqref{alphaundeta} holds true, note that 
\begin{align*}
\alpha_{y,z}= P_{\omega_i}\big(X(i-1)=y,X(i)=z, \tau\geq t, X(t) = x\big)/S_i.
\end{align*}
To compute the expectation of the r.h.s. of \eqref{alphaundeta}, consider, for $i$ fixed, the sigma algebra 
\[\mathcal F^*\coloneqq \sigma\big(\oy_i(s)\colon s< t,y\in\Z^d\big)\]
From our choice of $\omega_i$ we clearly have $\mathcal F_{i-1}\subseteq \mathcal F^*$, and $\eta_{y,z}$ is independent of $\mathcal F^*$ while $\alpha_{y,z}$ is $\mathcal F^*$ measurable. So using Jensen's inequality we obtain
\[\E\Big[\Big(\frac{S}{S_i}\Big)^{-\delta}\Big|\mathcal F^*\Big]\leq \sum_{y,z}\alpha_{y,z}\E\big[\eta_{y,z}^{-\delta}\;\big]=\sum_{y,z}\alpha_{y,z}\E\Big[(P_\omega(\tau\geq 1|X(1)=z-y)^{-\delta}\Big]\]
By Proposition \ref{prop:unifMom} we have
\[\sup_{y,z}\E\Big[P_\omega\big(\tau\geq 1\big|X(1)=z-y\big)^{-\delta}\Big]=c<\infty\]
and therefore
\[\E\Big[\Big(\frac S{S_i}\Big)^{-\delta}\Big|\mathcal F_{i-1}\Big]\leq c\;\E\Big[\sum_{y,z}\alpha_{y,z}\Big|\mathcal F_{i-1}\Big]=c.\]
\end{proof}
  
\subsection{Proof of Proposition \ref{prop:propSTilde}}\label{sec:stilde}

Equipped with this concentration inequality we can now prove Proposition \ref{prop:propSTilde}. We follow the proof of Proposition 2.4 in \cite{carmona}.
\begin{proof}[Proof of Proposition \ref{prop:propSTilde}]
We start with \eqref{eq:finirate}, where we first argue that it is enough to assume $t\in\N$. We take any $t>0$ and set $s:=\lfloor t\rfloor$. Then
\begin{align}
\E[\log \widetilde S(t)]\geq &\E[\log \widetilde S(s)]+\E[\log P_{\omega}(\tau\geq t-s,X(t-s)=0)]\nonumber\\
\geq &\E[\log \widetilde S(s)]-\frac 1\delta\log \sup_{r\in[0,1]}\E\Big[\big(P_{\omega}(\tau\geq r,X(r)=0)\big)^{-\delta}\Big]\label{eq:tinn}.
\end{align}
Here we take $\delta\in(0,\frac 12)$, and for the second line we used Jensen's inequality. The second term in \eqref{eq:tinn} is finite by \eqref{eq:expmom2}. For the first term we have
\begin{align*}
\E[\log \widetilde S(s)]&\geq \E\Big[\log \prod_{i=1}^sP^{(i)}_\omega(\tau\geq i,X(i)=0)\Big]\\
&=s\log P(X(1)=0)+s\;\E[\log P_\omega(\tau\geq 1|X(1)=0)].
\end{align*}
where $P^{(i)}$ is the law of a random walk started at time $i-1$ at the origin. Note that the first term in the last line does not depend on $\omega$. But for all $\delta\in(0,1)$ we have
\begin{align*}\E[\log P_\omega(\tau\geq 1|X(1)=0)]&= -\frac 1\delta\E\big[\log \big(P_\omega(\tau\geq 1|X(1)=0)^{-\delta}\big)\big]\\
&\geq -\frac 1\delta \log\E\big[P_\omega(\tau\geq 1|X(1)=0)^{-\delta}\big]>-\infty,
\end{align*}
where we used Jensen's inequality, and the integrability follows from Proposition \ref{prop:unifMom}.

From \eqref{eq:decayat0} and the concentration inequality \eqref{eq:conc} we obtain \eqref{eq:almostsuredecayat0} by a simple Borel-Cantelli argument. Now to prove \eqref{eq:decayat0}, we remark that the existence of the limit 
\[\widetilde p(\kappa)=\lim_{t\to\infty}\frac 1t\E[\log \widetilde S(t)]\]
can be shown by subadditivity as usual, but this is not even necessary for our claim. Clearly we have
\[\limsup_{t\to\infty}\frac 1t\E[\log \widetilde S(t)]\leq p(\kappa).\]
We now prove the other direction, where by \eqref{eq:tinn} it is enough to consider $t\in\N$. Note that for any $x\in\Z^d$ we have
\[P_\omega(\tau\geq 2t, X(2t)=0)\geq P_\omega(\tau\geq t, X(t)=x)P^{t,x}(\tau\geq 2t, X(2t)=0).\]
Since $P^{t,x}(\tau\geq 2t, X(2t)=0)$ has the same law as $P_\omega(\tau\geq t, X(t)=x)$, we conclude 
that
\begin{equation}\label{compxando}
\E[\log S(2t,0)]\geq 2\E[\log S(t,x)].
\end{equation}
For $\gamma>0$ we consider a box $B_t\coloneqq \{x\in\Z^d\colon \|x\|\leq \gamma t\}$ and the event $A_t\coloneqq \{X(t) \in B_t \}$. Using standard large deviation techniques, we can choose $\gamma$ large enough such that
\[\log P(A_t^c)<tp(\kappa)\quad  \forall t\geq t_0.\]
Consequently we have
\begin{equation}\label{mitoderohnebox}
p(\kappa)=\lim_{t\to\infty}\frac 1t \E[\log P_\omega(\tau\geq t)]=\lim_{t\to\infty}\frac 1t \E[\log P_\omega(\tau\geq t, A_t)]\, .
\end{equation}
Take now $\eps\coloneqq t^{-\frac 34}$ and apply the fractional moments method:
\begin{align}
\E[\log P_\omega(\tau\geq t,A_t)] &= \frac 1\eps \E\big[\log \big(P_\omega(\tau\geq t,A_t)^\eps\big)\big]\notag\\
&\leq \frac 1\eps \log \E\big[P_\omega(\tau\geq t,A_t)^\eps\big] = \frac 1\eps \log \E\Big[\Big(\sum_{x\in B_t} S(t,x)\Big)^\eps\Big]\label{jensenfirst}\\
&\leq \frac 1\eps \log \E\Big[\sum_{x\in B_t} S(t,x)^\eps\Big]\label{andnowsu}\\
&=\frac 1\eps\log \sum_{x\in B_t}\E\big[e^{\eps (\log S(t,x)-\E[\log S(t,x)])}\big]e^{\eps \E[\log S(t,x)]},\label{jetztvergl} 
\end{align}
where we get \eqref{jensenfirst} from Jensen's inequality, and the inequality in \eqref{andnowsu} comes
from the general estimate $\left(\sum_{j=1}^N a_j\right)^\eps \leq \sum_{j=1}^N a_j^\eps$ for nonnegative $a_1, \ldots ,a_N$ and $0<\eps<1$. For the left factor of the summands in \eqref{jetztvergl} we compute, using \eqref{eq:conc},
\begin{align*}
&{ } \E\Big[\exp\big(\eps (\log S(t,x)-\E[\log S(t,x)]\big)\Big]\\
&\leq 1+\int_1^\infty Q\Big(\big|\log S(t,x)-\E[\log S(t,x)]\big|>t^{\frac 34}\log u\Big)\dd u\\
&\leq 1+2\int_1^\infty e^{-Ct^{\frac 12}(\log u)^2 }\dd u\coloneqq c(t).
\end{align*} 
Then we are left with
\begin{align*}\E[\log P_\omega(\tau\geq t, A_t)]&\leq \frac 1\eps \log c(t) + \frac 1\eps \log \sum_{x\in B_t}e^{\eps \E[\log S(t,x)]}\\
&\leq \frac 1\eps \log c(t) + \frac 1\eps \log |B_t| + \frac 12 \E[\log S(2t,0)],
\end{align*}
where we have used \eqref{compxando}. Dividing by $t$ and taking limits, taking into account \eqref{mitoderohnebox}, we obtain
\[ \liminf_{t\to\infty}\frac{1}{2t} \E[\log S(2t,0)]\geq p(\kappa).\]
\end{proof}


\section{The critical case}\label{sec:critical}

\subsection{Proof of Theorem \ref{thm:main} in the critical case}\label{sec:outline}

In this section we apply the technique going back to \cite{criticalcontact}, where it was used to show that the critical contact process dies out. We consider the critical process and assume that it survives, showing that this leads to a contradiction. 

For this we find a supercritical oriented site percolation process induced by the branching process in such a way that an infinite cluster in the percolation implies global survival of the branching process. In this coupling the event that a site is open can be decided by considering local events of the branching process, i.e. an event that only depends on a finite space-time box. The probability of this local event therefore depends continuously on the parameters of the model, so that the comparison to supercritical percolation still holds true if we push the parameters slightly into the subcritical phase. Since we know that the process dies out in this case, we have a contradiction.

This technique was also used in \cite{garetmarchand} for a discrete time, non-degenerate version of our model.

\begin{proof}[Proof of Theorem \ref{thm:main} (critical case)]
Fix $\kappa$ and $\lambda$ 
such that 
\begin{equation}\label{eq:gleichnull}\lambda(m-1)+p(\kappa)=0.\end{equation}
At the same time assume that
\begin{equation}\label{eq:survival}\P^{\kappa,\lambda}(Z\text{ survives}) >0.\end{equation}
For the contradiction we first consider the process in an environment with a higher disaster rate, making the process subcritical: Let us introduce the rate at which disasters appear as a new parameter of the model (until now, it was fixed to be $1$). Denote by $Q^{\alpha}$ the law such that $\big(\ox\big)_{x\in\Z^d}$ is a collection of independent Poisson processes of rate $\alpha>0$, and write $\P^{\alpha, \kappa,\lambda}$ for the annealed measure $Q^\alpha\otimes P^{\kappa,\lambda}_\omega$. Let $p(\alpha,\kappa)$ be the survival rate of a single particle in this environment (defined as in \eqref{pdef} but with an environment with disaster rate $\alpha$). We show at the end of this section that for any $\delta>0$ we have 
\begin{equation}
\label{strictforalpha}
\lambda(m-1)+p(1+\delta,\kappa)<0.
\end{equation}
Now using the same arguments as in the proof of the subcritical part of Theorem \ref{thm:main}, \eqref{strictforalpha} implies that for all $\delta>0$ we have
\begin{equation}
\label{eq:null}\P^{1+\delta,\kappa, \lambda}(Z\text{ survives}) =0.
\end{equation}
The contradiction will come from a coupling with oriented percolation, showing that for $\delta$ small enough we have
\begin{equation}\P^{1+\delta,\kappa, \lambda}(Z\text{ survives}) > 0.\end{equation}

Consider a box $D_n\coloneqq\{-n,...n\}^d$. Recalling the notation from Section \ref{sec:notation}, we consider for $s,L,T\in\R^+$, $x,y\in\Z^d$ and $n,S\in\N$ the event
\begin{align}\label{eq:deflocal}
A^{s, y}(L,T,n,S)\coloneqq \left\{\begin{matrix}\exists x\in\{L,...,3L\}\times\{-L,...,L\}^{d-1}, t\in[5T, 6T] \\ \text{ such that }(x+D_n,S^2)\leq 
Z^{\{s\}\times(y+D_n,S^2)}_{\{-5L,...,5L\}\times \{-3L,...,3L\}^{d-1}}(t).\end{matrix}\right\}
\end{align}
(The reason to use $S^2$ on the r.h.s. of \eqref{eq:deflocal} will become clear later).
In words, $A^{0,0}= A^{0,0}(L,T,n, S)$ is the event that starting from configuration $(D_n,S^2)$ at time $0$, those particles will propagate such that at some time $t\in[5T,6T]$ we find a copy $x+D_n$ of $D_n$ where again every site is occupied by at least $S^2$ particles. 
Because we consider the truncated process, the event has to be achieved by particles which do not leave a certain space-time box. 

We now state the key proposition which says that under the assumption \eqref{eq:survival}, we can make the probability for (an auxiliary version of) $A^{0,0}$ arbitrarily large:
\begin{proposition}\label{prop:prop1}
Assume (\ref{eq:survival}). For every $\eps>0$ there exist $L, T>0$ and $n,S\in\N$ such that
\begin{equation}\label{eq:concl1}\P^{1,\kappa,\lambda}\left(\;\begin{matrix}\exists x\in\{L+n,...,2L+n\}\times\{-L,...,L\}^{d-1}, t\in[T, 2T] \\ \text{ such that } (x+D_n,S^2)\leq Z^{(D_n,S^2)}_{\{-L,...,3L\}\times\{-L,...,L\}^{d-1}}(t)\end{matrix}\;\right)> 1-\eps.\end{equation}
\end{proposition}

We will use this to give an estimate for the probability of $A^{s,y}$ that holds uniformly for all $s$ and $y$ in some space-time box:

\begin{proposition}\label{prop:prop2}
Assume (\ref{eq:survival}). For every $\eps'>0$ there exist $L', T'>0$ and $n,S\in\N$ such that
\[\inf\Big\{\P^{1,\kappa,\lambda}\left(A^{s,y}(L',T',n,S)\right)\colon s\in[0,T'],y\in\{-L',...,L'\}^d\Big\}> 1-\eps'.\]
\end{proposition}
Note that $A^{s,y}$ is a local event, i.e. it depends only on the process in some finite space-time box. Therefore its probability depends continuously on the parameters, and we get the following 
\begin{corollary}\label{cor:local}
Assume \eqref{eq:survival}. For every $\eps>0$ there exists $L,T>0$ as well as $n,S\in\N$ and $\delta>0$ such that
\begin{equation*}
\inf\Big\{\P^{1+\delta,\kappa,\lambda}\left(A^{s,y}(L,T,n,S)\right)\colon s\in[0,T],y\in\{-L,...,L\}^d\Big\}> 1-\eps.
\end{equation*}
\end{corollary}

We will now argue that Corollary \ref{cor:local} ensures that the process survives with positive probability by a comparison to oriented percolation on $\N^2$. We follow the arguments from chapter I.2 in \cite{liggett}. 

We call a path $(k_0,l_0),...,(k_m,l_m)$ in $\N^2$ an \textbf{oriented path} if for all $i=0,...,m-1$ we have $k_{i+1}=k_i+1$ and either $l_{i+1}=l_i$ or $l_{i+1}=l_i+1$. Fix $L, T, n$ and $S$. Let us call the point $(k,l)\in\N^2$ \textbf{occupied} if there exist $(t,x)$ with 
\begin{align*}
(x+D_n,S^2)\leq Z^{(D_n,S^2)}(t)
\end{align*}
and such that $(t,x)$ is in the space-time box
\begin{align}\label{eq:position}
\big[5Tk,5T(k+1)\big)\times \big\{L(-2k+4l-1),...,L(-2k+4l+1)\big\}\times\big\{-L,...,L\big\}^{d-1}.
\end{align}

Finally we call a point $(k,l)$ \textbf{open} if there is an oriented path $(k_0,l_0),...,(k_m,l_m)$ with $(k_0,l_0)=(0,0)$ and $(k_m,l_m)=(k,l)$, and such that $(k_i,l_i)$ is occupied for all $i=1,...,m$. We write 
\begin{align*}
\eta(k,l)=1
\end{align*}
if $(k,l)$ is open, and $\eta(k,l)=0$ otherwise. 

This defines a random process $(\eta(k,l))_{(k,l)\in\N^2}\in\{0,1\}^{\N^2}$ from every realization of the process $(Z^{(D_n,S^2)}(t))_{t\geq 0}$, and an easy observation is that if $\eta(k,l)=1$ for infinitely many points $(k,l)$, then the original process must have survived. The next proposition compares $(\eta(k,l))_{(k,l)\in\N^2}$ to an independent oriented site percolation 
\begin{align*}
(\widetilde\eta(k,l))_{(k,l)\in\N^2}.
\end{align*}
That is, for $(\widetilde\eta(k,l))_{(k,l)\in\N^2}$ every point $(k,l)\neq (0,0)$ is occupied independently with probability $p\in(0,1)$, and we set 
\begin{align*}
\widetilde \eta(k,l):=1
\end{align*}
if $(k,l)$ is reachable from $(0,0)$ along an oriented path of occupied points, and $\widetilde \eta(k,l)=0$ otherwise. For two realizations $(\eta(k,l))_{(k,l)}$ and $(\widetilde\eta(k,l))_{(k,l)}$, we will say that $\eta$ \textbf{dominates} $\widetilde\eta$ if $\eta(k,l)\geq\widetilde\eta(k,l)$ holds for every $(k,l)$.

\begin{proposition}\label{prop:percolation}
Assume \eqref{eq:survival}. For every $p\in(0,1)$, we can find values $L,T>0, n\in\N$, $ S\in\N$ and $\delta>0$ such that $(\eta(k,l))_{(k,l)\in\N^2}$ induced by 
\begin{align*}
\big(Z^{(D_n,S^2)}(t)\big)_{t\geq 0}
\end{align*}
with parameters $(1+\delta,\kappa,\lambda)$ can be coupled with $(\widetilde\eta(k,l))_{(k,l)\in\N^2}$ in such a way that $\eta$ dominates $\widetilde\eta$ with probability one.
\end{proposition}
This now gives the contradiction, since we can find $p\in(0,1)$ such that 
\begin{align*}
\P\big(\widetilde\eta(k,l)=1\text{ for infinitely many }(k,l)\big)>0.
\end{align*}
Proposition \ref{prop:percolation} then implies that 
\begin{align*}
\P^{1+\delta,\kappa,\lambda}\big(Z^{(D_n,S^2)}\text{ survives}\big)>0\quad\text{ for some }\delta>0.
\end{align*}
Clearly this also implies $\P^{1+\delta,\kappa,\lambda}(Z^{\{0\}}\text{ survives})>0$ and therefore contradicts \eqref{eq:null}.
\end{proof}

\begin{proof}[Proof of \eqref{strictforalpha}]
Let $\omega_\alpha$ and $\omega_\beta$ be independent environments of disaster rates $\alpha$ and $\beta$, respectively, and note that $\omega_{\alpha+\beta}\isd \omega_\alpha+\omega_{\beta}$. Here, $\omega_\alpha+\omega_{\beta}$ denotes the environment which contains both the disasters 
of $\omega_\alpha$ and of $\omega_\beta$.
We write $\tau_\alpha$ and $\tau_\beta$ for the extinction times in $\omega_\alpha$ and $\omega_\beta$ and get, using \eqref{expectaswell}, that
\begin{align*}
 p(\alpha+\beta,\kappa) &= \lim_{t\to\infty}\frac 1t\E \big[\log P_{\omega_\alpha + \omega_\beta}(\tau_\alpha\wedge\tau_\beta \geq t ) \big]
\\
&= p(\alpha,\kappa)+  \lim_{t\to\infty}\frac 1t\E[\log P_{\omega_\alpha+\omega_\beta}(\tau_\beta \geq t |\tau_\alpha\geq t) ]\\
&\leq p(\alpha,\kappa)+\lim_{t\to\infty}\frac 1t\log \E\big[P_{\omega_\alpha+\omega_\beta}(\tau_\beta\geq t|\tau_\alpha\geq t)\big]=p(\alpha,\kappa)-\beta,
\end{align*}
where the last inequality follows from Jensen's inequality. 
\end{proof}

\begin{proof}[Proof of Corollary \ref{cor:local}]
From Proposition \ref{prop:prop2} we get $L,T$, $n$ and $S$ such that 
\[\inf_{s\in[0,T],y\in\{-L,...,L\}^d}\P^{1,\kappa,\lambda}\left(A^{s,y}(L,T,n, S)\right)> 1-\frac \eps2.\]
Using the notation from the previous claim, we note that 
\begin{align*}
P^{\kappa,\lambda}_{\omega_1+\omega_\delta}(A^{s,y}(L,T,n, S))= P^{\kappa,\lambda}_{\omega_1}(A^{s,y}(L,T,n, S))
\end{align*}
on the event $E\coloneqq \big\{\omega_\delta\text{ has no disasters in }[0,6T]\times \{-5L,...,5L\}\times\{-3L,...,3L\}^{d-1}\big\}$. But the $Q$-probability of $E$ goes to $1$ for $\delta\to 0$, and therefore the claim follows from 
\begin{align*}
\inf_{s\in[0,T],y\in\{-L,...,L\}^d}\P^{1+\delta,\kappa,\lambda}\left(A^{s,y}(L,T,n, S)\right)\geq \Big(1-\frac \eps2\Big)Q(E)\, .
\end{align*}
\end{proof}

\begin{proof}[Proof of Proposition \ref{prop:percolation}]
We construct $(\eta(k,l))_{(k,l)\in\N^2}$ recursively. Assume that we have $\{\eta(k,l):k\leq K,l\leq K\}$ for some $K\in\N$. Since $\eta(K+1,l)=0$ unless either $\eta(K,l)=1$ or $\eta(K,l-1)=1$, we assume that at least one of the latter random variables has the value $1$. So we have 
\begin{align*}
(x+D_n,S^2)\leq Z^{(D_n,S^2)}(t)
\end{align*}
for some $(t,x)$ in either the space-time box from \eqref{eq:position} (with $K= k$) or satisfying
\begin{align*}
(t,x)\in  \big[5TK,5T(K+1)\big)\times \big\{L(-2K+4l-5),...,L(-2K+4l-3)\big\}\times\big\{-L,...,L\big\}^{d-1}.
\end{align*}
Clearly $\eta(K+1,l)=1$ holds if we find $(t',x')$ such that
\begin{align*}
(D_n+x',S^2)\leq Z^{\{t\}\times(x+D_n,S^2)}(t')
\end{align*}
where $t'\in [5T(K+1),5T(K+2))$ and
\begin{align*}
x'\in \big\{L(-2(K+1)+4l-1),...,L(-2(K+1)+4l+1)\big\}\times\big\{-L,...,L\big\}^{d-1}.
\end{align*}
Corollary \ref{cor:local} shows that for any $\eps>0$, we can choose $L,T>0,n\in\N$, $s \in \N$ and $\delta>0$ such that this happens with probability at least $1-\eps$, and it is clear that this probability does not depend on $K$ or $l$. 

So we have constructed a percolation $(\eta(k,l))_{(k,l)\in\N^2}$ where each point is open with high probability, but not independently. To address this we define a distance between two sets $S_1,S_2\subseteq\N$ by 
\begin{align*}
d(S_1,S_2):=\inf\{|x_1-x_2|:x_1\in S_1,x_2\in S_2\}
\end{align*}
We notice that the restriction to a truncated process in Corollary \ref{cor:local} ensures that the percolation is $2$-dependent. This means that conditioned on $\{\eta(k,l)\colon k\leq K\}$, the collections $(\eta(K+1,l))_{l\in S_1}$ and $(\eta(K+1,l))_{l\in S_2}$ are independent for any sets $S_1,S_2\subseteq \N$ with $d(S_1,S_2)> 2$. 

Theorem B26 in \cite{liggett} then ensures that we can couple $(\eta(K+1,l))_{l\leq K+1}$ with an independent family of Bernoulli random variables $(\widetilde\eta(K+1,l))_{l\leq K+1}$ such that $\eta$ dominates $\widetilde\eta$, and such that $\widetilde \eta(K+1,l)=1$ holds with probability at least $(1-\sqrt[5]{\eps})^2$ if either $\eta(K,l-1)=1$ or $\eta(K,l)=1$.
\end{proof}

It is clear that the key step is to prove Proposition \ref{prop:prop1} and Proposition \ref{prop:prop2} about the local events $A^{s,y}(L,T,n, S)$. For this we consider the numbers $N$ resp. $M$ of particles leaving a space-time box through the top resp. the faces, rigorously defined in Section \ref{sec:fkg}. Well established techniques for branching processes show that if the process survives we can expect $N+M$ to be large, which we show in some technical lemmas in Section \ref{sec:lemmas}. 

The construction we outlined before requires us to have more control on where exactly those particles exit the box, so we let $N(u,\theta)$ count the number of particles exiting through the orthant described by $\theta$ of the face in direction $u$, and we use $M(u,\theta)$ for the number of particles exiting through the corresponding orthant of the top. The formal definition is again deferred until Section \ref{sec:fkg}. Using an FKG inequality we obtain in Section \ref{sec:fkg} that we can expect all the $N(e_i,\theta)$ and $M(u,\theta)$ to be large on the event of survival, at least if we increase the number $S^2$ of particles initially on each site of $D_n$.

Having done all of this we can finally prove Propositions \ref{prop:prop1} and \ref{prop:prop2} in Section \ref{sec:props}, finishing the proof of Theorem \ref{thm:main}.


\subsection{Some technical lemmas}\label{sec:lemmas}
Recall that we have fixed $\lambda$, $\kappa$ and $q$ such that (\ref{eq:survival}) holds. We first show that we can make the survival probability arbitrarily close to $1$ by enlarging the set of initially occupied sites. This is part (i) of Lemma \ref{lem:lemma2} below.

Part (ii) concerns particles that survive locally until time $1$ by using only two sites. We obtain that (with high probability) many particles will achieve this if we start with a large enough number of particles at the origin.

Part (iii) shows that with high probability, starting from $N$ particles occupying the origin, at time $1$ we end up with a configuration where every site of $D_n$ is occupied by many particles. We need this to be a local event, so we restrict ourself to particles that do not leave certain boxes.



\begin{lemma}\label{lem:lemma2}
\begin{enumerate}
\item[(i)] For every $\eps>0$ there is $n\in\N$ with
\[\P\big(Z^{D_n}\text{ survives}\big)> 1-\eps.\]
\item[(ii)] Recall \eqref{particlesat0}. For every $\eps>0$ and $M\in\N$, there is an $N\in\N$ such that
\[\P\Big(\big|Z^{(\{0\},N)}_{\{0,e_1\}}(1)\cap\{0\}\big|\geq M\Big)>1-\eps.\]
\item[(iii)] Recall \eqref{substeil}. For every $\eps>0$ and $n,S\in\N$, there is an $N\in\N$ such that
\[\min\left\{\P\big((ne_1+D_n,S)\leq Z^{(\{0\},N)}_{ne_1+D_n}(1)\big),\P\big((D_n,S)\leq Z^{(\{0\},N)}_{D_n}(1)\big)\right\}>1-\eps\, .\]
\end{enumerate}

\end{lemma}

\begin{proof}

\textbf{Part (i)}: Define a collection $(Y_x)_{x \in \Z^d}$ with $Y_x\coloneqq  \1\{|Z^{\{x\}}(t)|>0\;\forall t>0\}$. We have
\[\P\big(|Z^{D_n}(t)|>0 \;\forall t\big) =\P \Big(\sum\limits_{x \in D_n}Y_x > 0\Big) 
= \E\Big[P_\omega\Big(\sum\limits_{x \in D_n}Y_x > 0\Big)\Big]\]
Writing $S_n\coloneqq \sum_{x\in D_n}Y_x$ we have
\begin{equation}\label{markovineq}
P_\omega\big(S_n = 0\big) \leq P_\omega\left(|S_n- E_\omega[S_n]| \geq E_\omega[S_n] \right)
\leq \frac{\text{Var}_\omega(S_n)} {\big(E_\omega[S_n]\big)^2}
\end{equation}
Now, due to the spatial ergodic theorem (see Theorem 4.9 in \cite{liggettIPS}), we have $\frac{1}{|D_n|} E_\omega[S_n]\to\E\big[E_\omega[Y_0]\big] > 0$ for almost all $\omega$, while 
\[\frac{1}{|D_n|} \text{Var}_\omega\big(S_n\big)=\frac{1}{|D_n|}\sum\limits_{x \in D_n}\text{Var}_\omega(Y_x)\to \E\big[\text{Var}_\omega(Y_0)\big]\quad Q\text{-a.s.},\]
where we used the fact that $\{Y_x, x \in \Z^d\}$ are independent with respect to $P_\omega$. We conclude from \eqref{markovineq} that $P_\omega(S_n= 0) \to 0 $ almost surely and therefore $\P \left(S_n = 0\right) \to 0$ as well.

\textbf{Part (ii)}: For $v\in\N^*$ a node in our tree (recall Section \ref{sec:constr}), let $B(v)$ denote the event that $v$ 
\begin{itemize}
 \item does not branch before time $1$
 \item satisfies $X([0,1],v)\subseteq\{0,e_1\}$ and $X(1,v)=0$
 \item and is not killed by the environment until time $1$.
\end{itemize}
For any $\alpha\in(0,1]$ we let $A(\alpha)$ be the event 
\begin{align*}
A(\alpha)\coloneqq \{P_\omega(B(\emptyset))\geq \alpha\}.
\end{align*}
Note that the events $A(\alpha)$ are increasing as $\alpha\downarrow 0$ and that their union over all $\alpha\in(0,1]\cap\Q$ has probability $1$. So for any $\eta>0$ we can find $\alpha>0$ small enough that \begin{align*}
Q(A(\alpha))\geq 1-\eta
\end{align*}

Now starting with $N$ initial particles at the origin in an environment $\omega\in A(\alpha)$, the number of particles $v$ such that $B(v)$ occurs dominates the number of successes of a binomial random variable with $N$ trials and success probability $\alpha$. Clearly we can choose $N$ large enough such that 
\[P\big(\operatorname{Bin}(N,\alpha)\geq M\big)\geq 1-\eta.\]
Then we can conclude since
\[\P \Big(\Big|Z^{(\{0\},N)}_{\{0,e_1\}}(1)\cap\{0\}\Big|\geq M\Big)\geq Q(A(\alpha))P(\xi_N\geq M)\geq (1-\eta)^2\geq 1-\eps\]
holds for $\eta$ small enough.

\par \textbf{Part (iii)}: Let $\widetilde D_n$ be equal to either $D_n$ or $ne_1+D_n$. We fix an enumeration $\widetilde D_n=\{x_1,...,x_{(2n+1)^d}\}$ of the sites, and introduce the quantity
\[S(x)\coloneqq P_\omega\big(\tau\geq 1,X(1)=x,X([0,1])\subseteq \widetilde D_n\Big).\]
Here we use $P_\omega$ for the law of a single particle which does not branch and which is killed by the environment $\omega$ with $\tau$ denoting its extinction time. For $\alpha\in(0,1]$ we consider events 
\[A(\alpha)\coloneqq \Big\{\min\{ S(x):x\in \widetilde D_n\}\geq \alpha\Big\}.\]
Fix $\eta>0$. By the same argument as before we find that $Q(A(\alpha))\geq 1-\eta$ holds for some $\alpha>0$ small enough. We now choose $N\coloneqq m(2n+1)^d$ for some large $m$. Letting $W\subseteq \N^*$ denote the set of initial particles, we partition $W$ (deterministically) in such a way that
\[W=W_1\cupdot \dots \cupdot W_{(2n+1)^d}\quad\text{ with }|W_i|=m\;\forall i=1,...,(2n+1)^d.\]
Now for $w\in W_i$ let $B_i(w)$ be the indicator function of the event that the particle $w$ 
\begin{itemize}
\item does not branch before time $1$
\item satisfies $X([0,1],w)\subseteq \widetilde D_n$ and $X(1,w)=x_i$ and
\item is not killed by the environment until time $1$.
\end{itemize}
Let $B$ be the event
\begin{align*}
B\coloneqq\Big\{\sum_{w\in W_i}B_i(w)\geq S\text{ for all }i=1,...,|\widetilde D_n|\Big\}.
\end{align*}
Noticing that $P(B_i(w)=1)=e^{-\lambda }S(x_i)$ we conclude that for $\omega\in A(\alpha)$ we have
\[P_\omega(B)= \prod_{i=1}^{|\widetilde D_n|}P_\omega\Big(\sum{w\in W_i}B_i(w)\geq R\Big)\geq P\big(\operatorname{Bin}(m, e^{-\lambda }\alpha)\geq R\big)^{(2n+1)^d}\]
But now it is clear that we can choose $m$ large enough that $P_\omega(B)\geq 1-\eta$ on $A$, hence
\[\P\Big((\widetilde D_n,S)\leq Z_{\widetilde D_n}^{(\{0\},N)}(1)\Big)\geq \int_{A(\alpha)}P_\omega(B)\dd Q\geq (1-\eta)^2\geq 1-\eps\]
holds for $\eta$ small enough.

\end{proof}


In the following, we think of $A\subseteq \Z^d$ as a large set, so that $\{Z^A\text{ dies out}\}$ is an event of small probability.
In the next lemma we state the familiar property that survival can only happen if the number of particles goes to infinity. Looking at the process as a random tree embedded in space-time, this means  that there are many particles occupying the top of a space-time box. 
\begin{lemma}\label{lem:lemma3}
For every $A\subseteq\Z^d$ we have
\[\P\big(Z^{A}\text{ survives}\big)=\P\big(Z^{A}\text{ survives}, \lim_{t\to\infty}|Z^A(t)|=\infty \big).\]
\end{lemma}
\begin{proof}
Define constants 
\begin{equation}\label{alphadef}
\alpha\coloneqq Q \Big(\text{At least one disaster occurs at the origin before time }1\Big)=1-e^{-1}
\end{equation}
\begin{equation}\label{betadef}
\beta\coloneqq  P\Big(\big(Z^{\{0\}}(t)\big)_{0\leq t\leq 1}\text{ stays at the origin and does not branch}\Big)=e^{-\lambda-\kappa}.
\end{equation}
Let $\mathcal F_t$ be the sigma algebra generated by the environment, the branching times and the particle positions up to time $t$. Then for any $t$ we have
\[\P(Z^{A}\text{ dies out}|\mathcal F_t)\geq \alpha^{|Z^{A}(t)|}\beta^{|Z^{A}(t)|}.\]
Letting $t$ go to infinity, the left side converges to the indicator function $\1{\{Z^A\text{  dies out}\}}\in\{0,1\}$. However, if for some $K$ we have $|Z^A(t)|<K$ for arbitrarily large $t$, the limit inferior of the right hand side will be bounded away from $0$. Therefore the event
\[\big\{Z^A\text{ survives}, |Z^A(t)|<K\text{ for arbitrarily large }t\big\}\]
has probability $0$. Now
\begin{align*}
& \P\big(Z^{A}\text{ survives}, \limsup_{t\to\infty}|Z^A(t)|<\infty \big)\\
&= \lim_{K\to \infty} \P\big(Z^{A}\text{ survives}, |Z^A(t)|<K\text{ for arbitrarily large }t\big)=0.
\end{align*}
\end{proof}

We also need the following general result:
\begin{lemma}\label{lem:neu}
 Let $m,S\geq 1$ and consider random variables $(X_0,...,X_m)$ taking values in $\{0,1\}^{m+1}$ (not necessarily independent). Then
\begin{align*}
\prod_{i=0}^m P\big(X_i =0\big)^S\leq P(X_i=0\text{ for all }i)+\Big(\frac{m}{m+1}\Big)^{(m+1)S}.
\end{align*}
\end{lemma}
\begin{proof}
For $I\subseteq\m=\{0,...,m\}$ let us define 
\begin{align*}
p_I:=P(\{i:X_i=0\}=I).
\end{align*}
We need to show that
\begin{align}\label{eq:needtoshow}
\prod_{i=0}^m\Big(p_{\m}+\sum_{\{i\}\subseteq I\subsetneq\m}p_I\Big)^S\leq p_{\m}+\Big(\frac{m}{m+1}\Big)^{(m+1)S}.
\end{align}
Observe that the left hand side takes a maximum over all values of $\big(p_I:I\neq \m\big)$ at $p_{I}=0$ for $|I|<m$ and 
\begin{align*}
p_{\m\setminus\{i\}}=\frac{1-p_{\m}}{m+1}\quad\text{ for }i=0,...,m.
\end{align*}
Then \eqref{eq:needtoshow} reads
\begin{align*}
\Big(\frac{m+p_{\m}}{m+1}\Big)^{(m+1)S}\leq p_{\m}+\Big(\frac{m}{m+1}\Big)^{(m+1)S}.
\end{align*}
Since the function on the left hand side is convex in $p_{\m}$ while the right hand side is linear, the conclusion follows by checking that the inequality indeed holds for $p_{\m}$ equal to $0$ and to $1$.
\end{proof}

\subsection{Space-time boxes and an FKG-inequality}\label{sec:fkg}
Let us define the random variables mentioned at the end of Section \ref{sec:outline}. Note that we can think of the process $(Z^\eta(t))_{0\leq t\leq T}$ as a process in space-time, which we want to emphasize by writing 
\[[0,T\,]\times Z^\eta\coloneqq \big\{(t,v)\colon 0\leq t\leq T, v\in Z^\eta(t)\big\}\subseteq [0,T]\times\N^*.\]
For convenience we also define the sign of zero to be $1$, that is 
\begin{align}\label{eq:defsign}
\sign(x)\coloneqq \1_{x\geq 0}-\1_{x<0}\quad\text{ for }x\in\Z.
\end{align}
For $L\in\N$ and $T>0$ we now consider a \textbf{space-time box} $\B\subseteq \R\times\Z^d$ of the form
\[\B\coloneqq [0,T]\times \{-L,...,L\}^d.\]
We denote the \textbf{top} of this box by 
\begin{align*}
\T(L,T)\coloneqq \{T\}\times \{-L,...,L\}^d.
\end{align*}
We can divide $\T$ in the \textbf{left} and \textbf{right} parts $\T(L,T,1)$ and $\T(L,T,-1)$, given by
\begin{align*}
\T(L,T,u)\coloneqq&\big\{(T,x)\in T(L,T)\colon \sign x_1=u \big\}\quad\text{ for }u\in\{\pm 1 \}.
\end{align*}
Moreover let $\U\coloneqq \{\pm e_i:i=1,...,d\}$ and for $u\in \U$ let $\F(L,T,u)$ denote the \textbf{face} in direction $u$, given by
\begin{align*}
\F(L,T,u)&\coloneqq[0,T] \times\left( \{-L,...,L\}^{i-1}\times\{0\}\times\{-L,...,L\}^{d-i}+Lu\right) .
\end{align*}
We need to partition both the top and the sides even further: Let $\Theta:=\{\pm 1\}^{d-1}$ and note that for every $\theta\in\Theta$ and $u\in \{\pm 1\}$ we find an \textbf{orthant} given by
\[\T(L,T,u,\theta)\coloneqq \{(T,x_1,...,x_d)\in\T(L,T,u)\colon \sign x_j=\theta_{j-1}\;\forall j=2,...,d\}.\]
Similarly an orthant on the face $\F(L,T,\pm e_i)$ has the form
\begin{align*}
\O(L,T,\pm e_i,\theta)\coloneqq \Big\{(t,x_1,...,x_d)\in \F(L,T,u)\colon \sign x_j=\theta_j\; \forall j<i, \sign x_j=\theta_{j-1}\; \forall j>i\Big\}.
\end{align*}
We further denote the \textbf{boundary} of $\B$ by $\partial \B$, that is
\begin{align*}
\partial \B(L,T)\coloneqq \T(L,T) \cup \bigcup_{u\in\U} \F(L,T,u).
\end{align*}
Note that the bottom $\{0\} \times \{-L, \cdots , L\}$ of the box is not part of the boundary. For all these quantities we sometimes omit the dependence on $L$ and $T$ if it is clear from the context. See also Figure \ref{fig:bild} for an example in $d=2$.
\begin{figure}[h]
  \begin{minipage}[c]{0.62\textwidth}
  
\begin{tikzpicture}[decoration={markings,mark=at position 1 with {\arrow[scale=2,>=to]{>}};}]
\pgfmathsetmacro{\cubex}{2}
\pgfmathsetmacro{\cubey}{4}
\pgfmathsetmacro{\cubez}{2}
\pgfmathsetmacro{\horizonx}{\cubex}
\pgfmathsetmacro{\horizonz}{3}

\foreach \x in {-\horizonx,...,\horizonx}
{
  \draw[gray!60] (\x, 0, \horizonz+1) -- ++(0, 0, -2*\horizonz-2);
}

\foreach \z in {-\horizonz,...,\horizonz}
{
  \draw[gray!60] (-\horizonx-1, 0, \z) -- ++(2*\horizonx+2, 0, 0);
}

\draw[line width = .4mm] (\cubex,0,\cubez) -- ++(-2*\cubex,0,0) -- ++(0,\cubey,0) -- ++(2*\cubex,0,0) -- cycle;
\draw[line width = .4mm,dashed] (\cubex,0,-\cubez) -- ++(-2*\cubex,0,0) -- ++(0,\cubey,0);
\draw[line width = .4mm] (\cubex,0,\cubez) -- ++(0,0,-2*\cubez) -- ++(0,\cubey,0) -- ++(0,0,2*\cubez) -- cycle;
\draw[line width = .4mm,dashed] (-\cubex,0,\cubez) -- ++(0,0,-2*\cubez) -- ++(0,\cubey,0);
\draw[line width = .4mm] (\cubex,\cubey,\cubez) -- ++(-2*\cubex,0,0) -- ++(0,0,-2*\cubez) -- ++(2*\cubex,0,0) -- cycle;

\draw[line width = .2mm, postaction={decorate}] (-\horizonx-1,0,0) -- ++(2*\horizonx + 2,0,0)  node[label={[shift={(.3,-.6)}]$\mathbb Z^2$}] {};
\draw[line width = .2mm, postaction={decorate}] (0,0,-\horizonz-1) -- ++(0,0,2*\horizonz + 2);
\draw[line width = .2mm, postaction={decorate}] (0,0,0) -- ++(0,\cubey + 1.5,0) node[label={[shift={(.5,-.5)}]$\R^+$}] {};

\draw[fill=gray, opacity=.3] (-\cubex, 0, -\cubez) -- ++(0,0,2*\cubez) -- ++(0,\cubey,0) -- ++(0,0,-2*\cubez) -- cycle;
\draw[fill=gray, opacity=.3] (-\cubex,\cubey,-\cubez) -- ++(0,0,2*\cubez) -- ++(\cubex,0,0) -- ++(0,0,-2*\cubez) -- cycle;

\draw (-\cubex / 2,\cubey,0) -- ++ (-\cubex *.7, .1*\cubey ,0) node[label={[shift={(-.3,-.3)}]$\mathbb T(L,T,-1)$}] {};
\draw (-\cubex, \cubey / 2,0) -- ++ (-.6*\cubex,.1*\cubey ,0) node[label={[shift={(-.65,-.25)}]$\mathbb F(L,T,-e_1)$}] {};

\draw[fill=gray, opacity=.3] (\cubex, 0, \cubez) -- ++(0,0,-\cubez) -- ++(0,\cubey,0) -- ++(0,0,\cubez) -- cycle;
\draw[fill=gray, opacity=.3] (\cubex, \cubey, \cubez) -- ++(0,0,-\cubez) -- ++(-\cubex,0,0) -- ++(0,0,\cubez) --cycle;

\draw (\cubex / 2,\cubey,\cubez/2) -- ++ (-.35*\cubex , -1.2*\cubey,\cubez ) node[label={[shift={(.1,-.65)}]$\mathbb T(L,T,1,1)$}] {};
\draw (\cubex,\cubey / 2,\cubez/2) -- ++ (+\cubex * .15, -.7*\cubey,\cubez ) node[label={[shift={(.1,-.65)}]$\mathbb F(L,T,e_1,1)$}] {};

\draw[gray] (\cubex,0,0) -- ++ (0,\cubey,0) -- ++ (-2*\cubex,0,0) -- ++(0,-\cubey,0);
\draw[gray] (0,0,\cubez) -- ++ (0,\cubey,0) -- ++ (0,0,-2*\cubez) -- ++(0,-\cubey,0);
\end{tikzpicture}

  \end{minipage}\hfill
  \begin{minipage}[c]{0.38\textwidth}
    \caption{The space-time box $\B$ in the case $d=2$. The relevant part of the boundary is subdivided in the following way: The left and right parts $\T(-1)$ and $\T(1)$ of the top, and the faces $\F(e_1)$ and $\F(-e_1)$ in direction $e_1$. Each of these is again subdivided in $2$ orthants, denoted $\T(\pm 1,\pm 1)$ and $\F(\pm e_1,\pm 1)$. \label{fig:bild}}
  \end{minipage}
\end{figure}
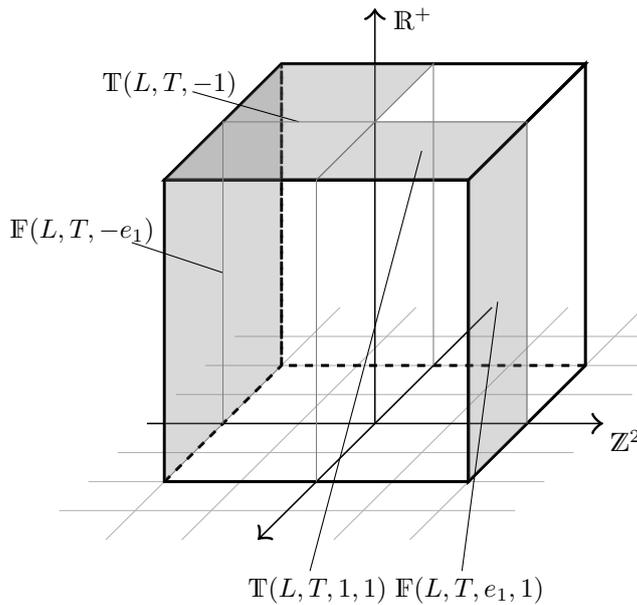
Let $\eta$ be a configuration as defined in Section \ref{sec:notation}. For $u\in\U$ and $\theta\in\Theta$ let 
\begin{align*}
N^\eta(L,T,u,\theta)
\end{align*}
count the number of particles leaving $\B$ through $\O(L,T,u,\theta)$. That is, $N^\eta(L,T,u,\theta)$ is the number of times such that a particle of $Z^\eta$ hits $\partial\B$ for the first time at some $(t,x)\in \O(L,T,u,\theta)$, formally defined as the cardinality of the set
\[\big\{(t,v)\in [0,T]\times Z^\eta\colon X(t, v)\in \O(L, T,u,\theta), X(s,v)\notin \partial \B\;\forall s<t\big\}.\]

Furthermore for $u\in\{\pm 1\}$ and $\theta\in\Theta$ let $M^\eta(L,T,u,\theta)$ count the particles exiting $\B$ through $\T(L,T,u,\theta)$, so that
\[M^\eta(L,T,u,\theta)\coloneqq \big|\big\{v\in Z^\eta(T)\colon X(T, v)\in \T(L,T,u,\theta),X(s, v)\notin \partial \B\;\forall s<T\big\}\big|.\]
We use $M^\eta$ and $N^\eta$ to refer to the vectors 
\begin{align*}
M^\eta(L,T,\cdot,\cdot)\in\N^{(2^d)}\quad\text{ and }\quad N^\eta(L,T,\cdot,\cdot)\in\N^{(d2^{d})}.
\end{align*}
Moreover we record the following shorthand notation for later use:
\begin{align}\label{eq:deftotalnumber}
\sum M^\eta(L,T)&\coloneqq\sum_{u\in\{\pm 1\},\theta\in\{-1,1\}^{d-1}} M^\eta(L,T,u,\theta)\\
\sum N^\eta(L,T)&\coloneqq \sum_{u\in\U,\theta\in\{-1,1\}^{d-1}} N^\eta(L,T,u,\theta).
\end{align}
We have the following FKG inequality.
\begin{theorem}\label{thm:fkg}
Let $\eta_1$ and $\eta_2$ be two configurations, and denote by $V^{\eta_1}$ and $\widetilde V^{\eta_2}$ two independent realizations of the process started from $\eta_1$ resp. $\eta_2$. We let $Z^{\eta_1}$, $M^{\eta_1}$ and $N^{\eta_1}$ (resp. $\widetilde Z^{\eta_2}$, $\widetilde M^{\eta_2}$ and $\widetilde N^{\eta_2}$) be defined as above for the processes started from $\eta_1$ (resp. $\eta_2$). Moreover let 
\begin{align*}
f,g\colon \N^{(2^d)}\times\N^{(d2^{d})}\to \R^+
\end{align*}
be increasing. Then
\begin{equation}\label{eq:fkg}
\E\Big[f\big(M^{\eta_1}, N^{\eta_1}\big)g\big(\widetilde M^{\eta_2}, \widetilde N^{\eta_2}\big)\Big]\geq \E\Big[f\big( M^{\eta_1}, N^{\eta_1})\Big]\E\Big[g( \widetilde M^{\eta_2}, \widetilde N^{\eta_2})\Big]\end{equation}
\end{theorem}
An intuitive explanation is that if many particles of $V^{\eta_1}$ survive and occupy any given orthant then this increases the chance that many particles of $\widetilde V^{\eta_2}$ are alive in any other orthant, since they are affected by the same disasters.


\begin{proof}[Proof of Theorem \ref{thm:fkg}]
We will show that for almost all realizations of $V^{\eta_1}$ and $\widetilde V^{\eta_2}$ we have
\begin{equation}\label{eq:darst2}
\begin{aligned}
&\int f\big( M^{\eta_1}(\omega), N^{\eta_1}(\omega)\big)g\big( \widetilde M^{\eta_2}(\omega), \widetilde N^{\eta_2}(\omega)\big)Q(\dd \omega)\\
\geq &\int f\big( M^{\eta_1}(\omega), N^{\eta_1}(\omega)\big)Q(\dd \omega) \int g\big( \widetilde M^{\eta_2}(\omega), \widetilde N^{\eta_2}(\omega)\big) Q(\dd \omega)
\end{aligned}
\end{equation}
Taking expectation with respect to the law of $V^{\eta_1}$ and $\widetilde V^{\eta_2}$ then yields the claim.  Think of the processes as trees, recalling Section \ref{sec:constr}.
Now conditioned on $V^{\eta_1}$ and $\widetilde V^{\eta_2}$ we can find $K\in\N$ and
\[0=U_0<U_1<...<U_{K}<U_{K+1}=T\]
such that both trees are constant on $[U_k,U_{k+1})$ for all $k=0,...,K$. That is, neither $V^{\eta_1}$ nor $\widetilde V^{\eta_2}$ jumps or branches in $[0,T]\setminus\{U_1,...,U_K\}$. Consider
\[\chi(k,x)\coloneqq \1\big\{\text{no disaster occurs at }x\text{ in the interval }[U_k,U_{k+1})\big\}.\]
Let $\mathcal G:=\sigma(\chi(k,x)\colon 0\leq k\leq K, x\in \Lambda)$ and note that $M^{\eta_1}, N^{\eta_1},\widetilde M^{\eta_2}$ and $\widetilde N^{\eta_2}$ are $\mathcal G$-measurable and increasing in $\chi$. Since $f$ and $g$ are increasing this means that both 
\begin{align*}
f( M^{\eta_1}, N^{\eta_1})\quad\text{ and }\quad g(\widetilde M^{\eta_2},\widetilde N^{\eta_2})
\end{align*}
are also increasing in $\chi$. Therefore (\ref{eq:darst2}) follows from the FKG inequality, see Corollary 2.12 in \cite{liggettIPS}. In this case the law of $\{\chi(k,x)\colon 0\leq k\leq K,x\in\Lambda\}$ trivially satisfies the FKG lattice condition since it is a product measure.
\end{proof}

We obtain the following
\begin{corollary}\label{cor:repair}
For any $L,K,K'\in\N, T>0$, any configuration $\eta$ and any $S\in\N$ we have
\begin{align}
\prod_{\theta\in\Theta,u\in\U}\P\big(M^{S\eta}(L,T,u,\theta)\leq K\big)&\leq \P\Big(\sum M^{\eta}(L,T)\leq d2^dK\Big)+(d2^d)^{-d2^dS}\label{eq:tasukatta1}\\
\prod_{\theta\in\Theta,u\in\{\pm 1\}}\P\big(N^{S\eta}(L,T,u,\theta)\leq K\big)&\leq \P\Big(\sum N^{\eta}(L,T)\leq 2^dK\Big)+(2^d)^{-2^dS}\label{eq:tasukatta2}
\end{align}
 and 
\begin{equation}
\begin{aligned}\label{eq:tasukatta3}
&\P\Big(\sum N^{S\eta}(L,T)\leq K\Big)\P\Big(\sum M^{S\eta}(L,T)\leq K'\Big)\\\leq& \P\Big(\sum M^{\eta}(L,T)+\sum N^{\eta}(L,T)\leq K+K'\Big)+4^{-S}
\end{aligned}
\end{equation}
\end{corollary}
\begin{proof}
We will show only the proof of \eqref{eq:tasukatta1} since the other claims follow in the same way. Let $I:=\U\times \Theta$, so that $|I|=d2^d$. Fix an environment $\omega$ and for $(u,\theta)\in I$ define
\begin{align*}
X_{u,\theta}:=\1\{M^{\eta}(L,T,u,\theta)> K\}.
\end{align*}
Now consider independent copies of the tree indexed by $I\times\{1,...,S\}$, each of which is started from configuration $\eta$ and evolves in the same environment. We use $X_{u,\theta,i}$ to denote the realization of $X_{u,\theta}$ corresponding to the tree $(u,\theta,i)\in I\times\{1,...,S\}$, which is now an independent family. Observe that
\begin{align*}
P_{\omega}\big(X_{u,\theta,i}=0\text{ for all }i=1,...,S\big)\geq P_{\omega}\big(M^{S\eta}(L,T,u,\theta)\leq K\big).
\end{align*}
Together with Lemma \ref{lem:neu} this implies
\begin{align*}
\prod_{u,\theta}P_{\omega}\big(M^{S\eta}(L,T,u,\theta)\leq K\big)&\leq P_{\omega}\big(M^\eta(L,T,u,\theta)\leq K\text{ for all }u,\theta\big)+(d2^d)^{-d2^dS}\\
&\leq P_{\omega}\Big(\sum M^\eta(L,T)\leq d2^dK\Big)+(d2^d)^{-d2^dS}
\end{align*}
Finally the claim follows by taking expectations, and applying Theorem \ref{thm:fkg} to the left hand side.
\end{proof}

The next lemma shows that we can make the probability on the right hand side of \eqref{eq:tasukatta3} arbitrarily small: That is, if the process survives then there will be many particles occupying the boundary of any space-time box:
\begin{lemma}\label{lem:lemma3part2}
Let $(T_j)_j$ and $(L_j)_j$ be two sequences increasing to infinity. 
Then for any $K>0$ and any configuration $\eta$ have
\[\limsup_{j\to\infty}\P\Big(\sum N^\eta(L_j,T_j)+\sum M^\eta(L_j,T_j)<K\Big)\leq \P\big(Z^\eta\text{ dies out }\big).\]
\end{lemma}
\begin{proof}
Let $\Lambda_j\coloneqq \{-L_j+1,\cdots ,L_j-1\}^d$, and consider the space-time box $\B_j\coloneqq [0,T_j]\times \Lambda_j$. We denote by $\mathcal{F}_{L_j,T_j}$ the sigma algebra generated by the environment in $\B_j$ as well as the branching times and positions of particles inside $\B_j$. We will consider the process of particles in $Z^\eta$ that have never left $\B_j$:
\begin{align*}E_j\coloneqq &\Big\{(s,v)\in[0,T]\times Z^\eta\colon \|X(s,v)\|_{\infty}=L_j, \|X(r,v)\|_\infty<L_j\;\text{ for all }r<s\Big\}\\
&\cup\;\Big\{(T,v)\in\{T\}\times Z^\eta\colon \|X(r,v)\|<L_j\;\text{ for all }r\leq T\Big\}.\end{align*}
Here $\|\cdot\|_\infty$ denotes the maximum norm. Note that $(s,v)\in E_j$ implies that the particle $v$ has just left $\B_j$ (for the first time) at time $s$, 
either through one of the sides or through the top. Clearly $E_j$ is $\mathcal F_{L_j,T_j}$ measurable and we have $|E_j|=N^\eta(L_j,T_j)+M^\eta(L_j,T_j)$. Now for $(s,v)\in[0,T]\times Z^\eta$ let $D(s,v)$ be the indicator function of the event that $v$ is killed because
\begin{itemize}
\item there is a disaster at $X(s,v)$ in the interval $[s,s+1]$
\item and $v$ has no branching times and no jumps in $[s,s+1]$.
\end{itemize}
Then $\P(D(s,v)=1)=\alpha\beta$ with the same $\alpha$ and $\beta$ as in \eqref{alphadef} and \eqref{betadef}. We can write
\begin{align}
\label{eq:diesout}\P(Z^\eta\text{ dies out}|\mathcal F_{L_j,T_j})\geq \P\big(D(s,v)=1\text{ for all }(s,v)\in E_j\big|\mathcal F_{L_j,T_j}\big)\geq \alpha^{|E_j|}\beta^{|E_j|}.
\end{align}
For the last estimate, note that for $(s,v)\in E_j$ the event $D(s,v)=1$ is independent of $\mathcal F_{L_j,T_j}$ and that for $(s_1,v)\neq(s_2,w)\in E_j$ we have
\[\P(D(s_1,v)=D(s_2,w)=1)\geq \P(D(s_1,v)=1)\P(D(s_2,w)=1).\]
Now the same argument as in the proof of Lemma \ref{lem:lemma3} applies: For $j\to\infty$ the left hand side of (\ref{eq:diesout}) converges to 
$\1\{Z^\eta\text{ dies out}\}$, while the right side will be bounded away from zero whenever $|E_j|\leq K$ for infinitely many $j$. Therefore we have
\[\limsup_{j\to\infty} \P(|E_j|<K)\leq \P(|E_j|\leq K\text{ i.o.})\leq \P(Z^\eta\text{ dies out}).\]
\end{proof}

\subsection{Proof of the key propositions}\label{sec:props}
We are now in a position to prove the missing Propositions from Section \ref{sec:outline}. Note that there we have only used Proposition \ref{prop:prop2}, however we obtain it by repeatedly applying Proposition \ref{prop:prop1}. For the proof of this first result we need to consider two cases depending on the value of $\eps$. Since $\eps$ will in turn depend on the value $\eps'$ in the second proposition, we choose to state those two cases in terms of $\eps'$ from the beginning: Given $\eps'> 0$ we choose $\eps>0$ such that
\begin{equation}\label{eq:defeps}
(1-\eps)^{10}\geq 1-\eps'.
\end{equation}
With this value of $\eps$ we can find $\delta>0$ such that
\begin{equation}\label{eq:defdelta}
\min\Big\{\Big(1-(3\delta)^{(2^{d})^{-1}}\Big)\Big(1-(2\delta)^{(d2^{d})^{-1}}\Big)(1-\delta)^3,1-3\delta\Big\}\geq 1-\eps.
\end{equation}
By Lemma \ref{lem:lemma2} we can find $n\in\N$ such that 
\begin{equation}\label{eq:defn}\P\big(Z^{D_n}\text{ survives}\big)\geq 1-\delta^2.\end{equation}
Moreover let $S$ be an integer such that 
\begin{align*}
\max\Big\{\Big(1-\frac{1}{d2^d}\Big)^{d2^dS},\Big(1-\frac{1}{2^d}\Big)^{2^dS},4^{-S}\Big\}\leq \frac{\delta^2}2.
\end{align*}
Now one of the following two statements will be true, and we prove both propositions separately in each case: 
\begin{align}
\forall L\in\N\text{ we have }&\P\big(Z^{(D_n,S)}_L\text{ survives}\big)<1-2\delta.\label{eq:case1}\tag{case 1}\\
\exists L\in\N\text{ such that }&\P\big(Z^{(D_n,S)}_L\text{ survives}\big)\geq 1-2\delta \label{eq:case2}\tag{case 2}.
\end{align}


\subsubsection{Proof in \ref{eq:case1}}
\begin{proof}[Proof of Proposition \ref{prop:prop1} in \ref{eq:case1}]

We first have to find a number $R\in\N$ that is large enough for our purposes: Let
\begin{equation}\label{eq:alpha}\alpha\coloneqq \min\Big\{\P\Big((ne_1+D_n,S^2)\leq Z^{\{0\}}_{ne_1+D_n}(1)\Big),P\Big((D_n,S^2)\leq Z^{\{0\}}_{D_n}(1)\Big)\Big\}>0.
\end{equation}
Then choose $R_1$ such that 
\begin{align*}
1-(1-\alpha)^{R_1}>1-\delta
\end{align*}
and set $R_2\coloneqq  R_1(4n)^d$. Note that this ensures that any set $A\subseteq\Z^d$ with $|A|\geq R_2$ contains a subset $A'\subseteq A$ with $|A'|\geq R_1$ and such that for every two sites $x\neq y\in A'$ we have $\|x-y\|_\infty\geq 4n$. By part (iii) of Lemma \ref{lem:lemma2} we find $R_3$ such that 
\[\min\Big\{\P\Big((ne_1+D_n,S^2)\leq Z^{(\{0\},R_3)}_{ne_1+D_n}(1)\Big),\P\Big((D_n,S^2)\leq Z^{(\{0\},R_3)}_{D_n}(1)\Big)\Big\}> 1-\delta.\]
Finally due to part (ii) of Lemma \ref{lem:lemma2} we can choose $R_4$ large enough that
\[\P\Big(\big|Z^{(\{0\},R_4)}_{\{0,e_1\}}(1)\cap\{0\}\big|\geq R_3\Big)>1-\delta.\]
Now set
\[R\coloneqq \big(\max\big\{R_1,R_2,n R_4\big\}\big)^2.\]
The next step is to find $L$ and $T$. From Lemma \ref{lem:lemma3} and the definition of $n$ we obtain
\[\lim_{T\to\infty}\lim_{L\to\infty} \P\Big(\sum M^{D_n}(L,T)>2^dR\Big)=\lim_{T\to\infty} \P(|Z^{D_n}(T)|>2^dR)\geq 1-\delta^2.\]
We can rewrite this by saying that for all $T\geq T_0$ there exists $L(T)$ with
\begin{equation}\label{eq:defLT}
\P\Big(\sum M^{(D_n,S)}(L,T)>2^dR\Big)\geq \P\Big(\sum M^{D_n}(L,T)>2^dR\Big)\geq 1-\delta\quad\forall L\geq L(T).
\end{equation}
That is, the probability that there are $2^dR$ particles at the top of a box $[0,T]\times \{-L,...,L\}^d$ can be made large by choosing $L$ and $T$ large enough. We want a similar result for the number of particles leaving through the sides of $[0,T]\times \{-L,...,L\}^d$. Using (\ref{eq:defLT}) and (\ref{eq:case1}) we can define two increasing sequences $(L_k)_{k\geq 0}$ and $(T_k)_{k\geq 0}$, starting with $T_0$ the value used for \eqref{eq:defLT} and $L_0\coloneqq L(T_0)+1$. For $k\geq 1$ we proceed by
\begin{align*}
 L_{k+1}\coloneqq &\max\big\{L_k+1,L(T_{k}+1)\big\}\\
 T_{k+1}\coloneqq &\inf\Big\{T>T_k:\P\Big(\sum M^{(D_n,S)}(L_{k+1},T)>2^dR\Big)<1-2\delta\Big\}.
\end{align*}
Since $T\mapsto \P(\sum M^{D_n}(L,T)>2^dR)$ is continuous we have
\begin{equation}\label{eq:defT}
\P\Big(\sum M^{(D_n,S)}(L_{k},T_k)\leq 2^dR\Big)=2\delta\quad\text{ for all }k.
\end{equation}
Note that our space-time box has $2^d$ orthants in the top and $d2^d$ orthants in the faces. We therefore apply Lemma \ref{lem:lemma3part2} with $K$ equal to $(1 +d)2^dR+1$ and the sequences $(L_k)_k$ and $(T_k)_k$ defined before. We find that there exists $k_0$ such that for all $k\geq k_0$ we have
\[\P\Big(\sum N^{D_n}(L_k,T_k)+\sum M^{D_n}(L_k,T_k)\leq (d+1)2^{d}R\Big) \leq \frac 32 \P\big(Z^{D_n}\text{ dies out}\big)\leq \frac {3}2\delta^2.\]
We set $L\coloneqq L_{k_0}$ and $T\coloneqq T_{k_0}$. Then we have
\begin{align*}
\frac 32\delta^2 &\geq \P\Big(\sum N^{D_n}(L,T)+\sum M^{D_n}(L,T)\leq (d+1)2^{d}R\Big)\\
&\geq \P\Big(\sum N^{(D_n,S)}(L,T) \leq d2^dR\Big)\P\Big(\sum M^{(D_n, S)}(L,T)\leq 2^dR\Big)-\frac {\delta^2}2.
\end{align*}
For the second inequality we have used \eqref{eq:tasukatta3} and the definition of $S$. Together with \eqref{eq:defT} we get
\begin{equation}\label{eq:forlater}
\P\Big(\sum N^{(D_n,S)}(L,T)\leq  d2^dR\Big)\leq \delta.
\end{equation}
Applying \eqref{eq:tasukatta2} together with the definition of $S$, and using the fact that by symmetry, the value of $\P\big(N^{(D_n,S^2)}(L,T,u,\theta)\leq R\big)$ does not depend on $\theta$ and $u$, we obtain
\begin{align}\label{eq:propNn}
\P\big(N^{(D_n,S^2)}(L,T,u,\theta)\leq R\big)^{d2^d}\leq \delta+\frac{\delta^2}2\leq 2\delta
\end{align}
On the other hand \eqref{eq:tasukatta1} together with \eqref{eq:defT} and the definition of $S$ shows 
\begin{align}\label{eq:propMn}
\P\big(M^{(D_n,S^2)}(L,T,u,\theta)\leq R\big)^{2^d}\leq 3\delta.
\end{align}

\begin{remark}
Clearly the probabilities in (\ref{eq:propNn}) and (\ref{eq:propMn}) do not depend on the choice of $\theta$ and $u$, a fact that we will use in the proof of Proposition \ref{prop:prop2}.
\end{remark}

Now we have to verify that the claim of proposition \ref{prop:prop1} is indeed satisfied with this choice of $L$ and $T$. That is, we need to bound the probability that we find a copy of $D_n$ shifted to the correct space-time location, and such that every site is occupied by at least $S^2$ particles of the truncated tree. We show that each of the following steps occurs with high probability, independent of the choice of $\theta\in\Theta$:
\begin{enumerate}
\item[1.]The tree $Z^{(D_n,S^2)}$ has many particles leaving through $\O(L,T,e_1,\theta)$.
\item[2.]There exist $(t,x)\in \O(L,T,e_1,\theta)$ such that the particles occupying $x$ at time $t$ grow into a fully occupied copy $\{t+1\}\times(x+ne_1+D_n,S^2)$ of $(D_n,S^2)$.
\item[3.]Consider now the box 
\begin{align*}
\goodoverline{\B}:=\left( [0,T]\times\{-L,...,L\}^d\right) +(t+1,x+ne_1)
\end{align*}
The tree growing from $\{t+1\}\times(x+ne_1+D_n,S^2)$ will have many descendants that leave through the top $\goodoverline{\T}(1,-\theta)$ of $\goodoverline\B$.
\item[4.]There is one particle at $(\goodoverline t,\goodoverline x)\in \goodoverline\T(1, -\theta)$ that grows into a new copy of the box $\{\goodoverline t+1\}\times(\goodoverline x+D_n,S^2)$, which now satisfies the necessary conditions.
\end{enumerate}
%
%
\textbf{First step}: We have shown this in (\ref{eq:propNn}).
\\\textbf{Second step}: This will follow from our choice of $R$. We need to consider the geometry of the set
\begin{align*}
\mathcal R\coloneqq \Big\{(t,x)\in \F(L,T,e_1,\theta)\colon \exists v\in Z^{(D_n,S^2)}_{L}(t)\text{ s.t. }x=X(t,v),X(s,v)\notin \partial \B\;\forall s<t\Big\}
\end{align*}
 of space-time-points where a particle leaves $[0,T]\times \{-L,...,L\}^d$ through the orthant $\O_+(L,T,\theta)$ for the first time. Observe that 
\begin{align*}
N^{(D_n,S^2)}(L,T,e_1,\theta)=|\mathcal R|
\end{align*}
Let $I$ be the (finite) index set
\[I\coloneqq \Big((\N)\times\{L\}\times (n\Z^{d-1})\Big)\cap \O(L,T,e_1,\theta).\]
Set $H\coloneqq [0,1]\times\{0\}\times\{0,...,n-1\}^{d-1}$, so that we obtain a tiling with
\[\O(L,T,e_1,\theta)\subseteq \bigcup_{(t_i,x_i)\in I}\big((t_i,x_i)+H\big).\]
On $\{N^{(D_n,S^2)}(L,T,e_1,\theta)>R\}$ at least one of the following statements will be true:
\begin{itemize}
\item There exist at least $\sqrt R$ distinct indices $(t,x)\in I$ such that 
\begin{equation}\label{eq:caseA}\tag{case A}
\mathcal R\cap \big((t,x)+H\big)\neq \emptyset. 
\end{equation}
\item There exists $(t_{0},x_{0})\in I$ such that 
\begin{equation}\label{eq:caseB}\tag{case B}
\big|\mathcal R\cap \big((t_{0},x_{0})+H\big)\big| \geq \sqrt R.
\end{equation}
\end{itemize}
For both cases we let $E_{t,v}$ be the indicator function of the event that $(t,v)\in\mathcal R$ grows into a shifted copy of $D_n$:
\[E_{t,v}\coloneqq \1\big\{\big(X(t,v)+D_n+ne_1,S^2\big)\leq Z^{\{t\}\times\{X(t,v)\}}_{x+ne_1+D_n}(1)\big\}.\]
In \textbf{(\ref{eq:caseA})} note that $\sqrt R\geq (4n)^d R_1$, so we can find at least $R_1$ distinct indices $(t_1,x_1)$, $...,(t_{R_1},x_{R_1})\in I$ such that 
\[|t_i-t_j|\geq 2\quad\text{ and }\quad\|x_i-x_j\|_\infty\geq 4n\quad\text{ holds for all }i\neq j.\]
Now choose (deterministically) some $(s_i,v_i)\in \mathcal R$ with 
\begin{align*}
(s_i,X(s_i,v_i))\in (t_i,x_i)+H.
\end{align*}
Because of the truncation the events $\{E_{s_i,v_i}=1\}$ and $\{E_{s_j,v_j}=1\}$ are independent for $i\neq j$. Moreover the probability that $E_{s_i,v_i}=1$ is at least $\alpha$, defined in (\ref{eq:alpha}). By our choice of $R_1$ we have
\[\P(E_{(s,v)}=1\text{ for some }(s,v)\in\mathcal R)>1-\delta.\]
In \textbf{(\ref{eq:caseB})} we find $y\in x_{0}+\{L\}\times\{0,...,n-1\}$ such that at least $\frac{\sqrt R}{n}\geq R_4$ particles arrive at $[t_{0},t_{0}+1]\times \{y\}$. Let $G$ be the event that 
\begin{itemize}
 \item at least $R_3$ of those particles survive until time $t_{0}+1$
 \item while not leaving the set $\{y,y+e_1\}$, 
 \item and occupying $y$ at time $t_{0}+1$.
\end{itemize}
By our choice of $R_4$ and part (ii) of Lemma \ref{lem:lemma2} we obtain
\[\P(G)\geq \P\Big(\Big|Z^{(\{0\},R_4)}_{\{0,e_1\}}(1)\cap\{0\}\Big|\geq R_3\Big)\geq 1-\delta.\]
Let now $G'$ be the event that at time $t_{0}+2$ every site of $y+ne_1+D_n$ is occupied by at least $S^2$ descendants of the particles occupying $y$ at time $t_{0}+1$. By our choice of $R_3$ and part (iii) of Lemma \ref{lem:lemma2} we find that
\[\P(G'|G)\geq \P\big((ne_1+D_n,S^2)\leq Z_{ne_1+D_n}^{\{t_0+1\}\times(\{y\},R_3)}(t_0+2)\big)\geq 1-\delta\]
Now combining both cases and (\ref{eq:propNn}) yields
\begin{equation}\label{eq:second}
\P\left(\begin{matrix}\exists x\in \{L+n\}\times \{-L,...,L\}^{d-1}, t\in[0,T+1]\\
\text{s. th.}(x+D_n,S^2)\leq Z_{\{-L,...,L+2n\}\times\{-L,...,L\}^{d-1}}^{(D_n,S^2)}(t)\end{matrix}\right)\geq \Big(1-(2\delta)^{(d2^{d})^{-1}}\Big)(1-\delta)^2
\end{equation}
\textbf{Third step}: We now write $\goodoverline{\P}$ for $\P$ conditioned on the event in (\ref{eq:second}), and denote the first such pair by $(t,x)$. From now on we consider the process 
\[\big(\goodoverline Z_L(s)\big)_{s\geq t}:=\Big(Z_{x+\{-L,...,L\}^{d}}^{\{t\}\times(x+D_n,S^2)}(s)\Big)_{s\geq t}\]
started from $\{t\}\times(x+D_n,S^2)$. Observe that under $\goodoverline P$, the process $\goodoverline Z_L$ is independent of the process up to time $t$. We consider a shifted space-time box 
\[\goodoverline\B\coloneqq (t,x)+[0,T]\times\{-L,...,L\}^d.\]
and let $\goodoverline M(e_1,\theta)$ resp. $\goodoverline M(-e_1,\theta)$ count the number particles of $\goodoverline Z_L$ that leave $\goodoverline\B$ through $\goodoverline \T(e_1,\theta)$ resp. $\goodoverline \T(-e_1,\theta)$. By (\ref{eq:propMn}) we have 
\[\goodoverline{\P}(\goodoverline M(e_1,-\theta)\geq R)\geq 1-(3\delta)^{2^{-d}}.\]
\textbf{Fourth step}: On the event $\{\goodoverline M(e_1,-\theta)\geq R\}$ one of the following two cases will occur:
\begin{align}
\Big|\Big\{\goodoverline x\in\goodoverline\T(e_1,-\theta):\big|\{\goodoverline x\}\cap\goodoverline Z_L(T)\big|>0\Big\}\Big|&\geq\sqrt R\label{eq:newCaseA}\tag{\ref{eq:caseA}'}\\
\exists\goodoverline {x_0}\in\goodoverline\T(e_1,-\theta)\quad\text{ such that }\quad \big|\{\goodoverline x_0\}\cap\goodoverline Z_L(T)\big|&\geq \sqrt R.\label{eq:newCaseB}\tag{\ref{eq:caseB}'}
\end{align}
In \textbf{(\ref{eq:newCaseA})} we note that $\sqrt R\geq (4n)^dR_1$, and thus we find at least $R_1$ sites $\goodoverline x_1,...,\goodoverline x_{R_1}$ in $\goodoverline\T(e_1,-\theta)$, each occupied by at least one particle, with the property that 
\begin{align*}
\|\goodoverline x_i-\goodoverline x_j\|_\infty\geq 2n+1\quad\text{ for all }i\neq j.
\end{align*}
For $\goodoverline x\in\goodoverline\T(e_1,-\theta)$ we let $\goodoverline E_x$ be the indicator function of the event 
\[
\big\{(\goodoverline x+D_n,S^2)\leq Z_{\goodoverline x+D_n}^{\{t+T\}\times\{\goodoverline x\}}(t+T+1)\big\}.
\]
Because of the truncation the events $\{\goodoverline E_{\goodoverline x_i}=1\}$ and $\{\goodoverline E_{\goodoverline x_j}=1\}$ are independent under $\goodoverline{\P}$ for all $i\neq j$. Since  $\goodoverline{\P}(\goodoverline E_x=1)\geq \alpha$ the definition of $R_1$ implies
\[\goodoverline{\P}(\goodoverline E_{\goodoverline x_i}=1\text{ for some }i=1,...,R_1)\geq 1-\delta.\]
Finally, in \textbf{(\ref{eq:newCaseB})} our choice of $R_3$ implies that 
\begin{align*}
\P\Big((\goodoverline x+D_n,S^2)\leq Z^{\{t+T\}\times(\{\goodoverline x\},R_3)}_{\goodoverline x+D_n}(t+T+1)\Big)\geq 1-\delta.
\end{align*}
We have shown that
\begin{align*}\label{eq:fourthh}\goodoverline{\P}\left(\begin{matrix}\exists \goodoverline x\in\goodoverline \T(e_1,-\theta)\text{ s. th. }(\goodoverline x+D_n,S^2)\leq \goodoverline Z_L(t+T+1)\end{matrix}\right)\geq \Big(1-(3\delta)^{(d2^d)^{-1}}\Big)(1-\delta).
\end{align*}
Since $(t+T+1,\goodoverline x)\in [T,2T]\times\{L+n,...,2L+n\}\times\{-L,...,L\}^{d-1}$, the claim now follows from this together with (\ref{eq:second}) and our choice of $\delta$.
\end{proof}

\begin{proof}[Proof of Proposition \ref{prop:prop2} in \ref{eq:case1}]
Set $L'\coloneqq 2L+n$ and $T'\coloneqq 2T$. Recall that in the previous proof we chose $\theta\in\Theta$ and $u\in\{\pm e_1\}$, and then bounded the probability of the event that
\begin{itemize}
\item we find $R$ particles in the orthant $\O(u,\theta)$ in (\ref{eq:propMn}).
\item starting from those particles, we again find $R$ particles in the orthant $\goodoverline \T(u,-\theta)$ of the top of a shifted box in (\ref{eq:propNn}).
\end{itemize}
We now repeatedly apply this result, each time making a convenient choice for $\theta$ and $u$. We start with $\big(s^{(0)}, y^{(0)}\big)\coloneqq (s, y)$ from the statement of the proposition. Having constructed $\big(s^{(0)}, y^{(0)}\big),...,\big(s^{(k)}, y^{(k)}\big)$ we choose 
\[\theta_{k+1}\coloneqq -\big(\sign y^{(k)}_2,...,\sign y^{(k)}_d\big)\in\Theta\]
and $u_{k+1}$ equal to $e_1$ until the first $k$ with $y^{(k)}_1\geq L'+L$, after which we alternate by setting $u_{i+1}\coloneqq -u_i$. By Proposition \ref{prop:prop1} we know that with probability at least $(1-\eps)$ we find $\big(s^{(k+1)}, y^{(k+1)}\big)$ such that 
\[(y^{(k+1)}+D_n,S^2)\leq Z_{y^{(k)}+\{-L,...,3L\}\times\{-L,...,L\}^{d-1}}^{\{s^{(k)}\}\times(y^{(k)}+D_n,S^2)}(s^{(k+1)}+s^{(k)})\]
Note that by our choice for $\theta_k$ and $u_k$ we have
\begin{itemize}
\item $|y^{(k)}_i|\leq 2L\leq L'$ for all $k\geq 0$ and $i=2,...,d$. 
\item $y^{(k)}_1\in \{L',...,3L'\}$ eventually: We achieve $y^{(k)}_1\geq L'+L$ after at most $4$ applications, and by alternating the sign of $u_i$ for $i\geq k$ we ensure $L'\leq y^{(i)}_1\leq 3L'$ for all $i\geq k$.
\item $s^{(i)}\in[5T',...,6T']$ for some $i\geq k$: After $4$ applications we have $s^{(i)}\in [4T,...,8T]=[2T',...,4T']$. Since $y^{(i)}$ remains in the target area, we can repeat the procedure until $s^{(i)}\in[5T',...,6T']$.
\end{itemize}
Note that this requires between $4$ and $10$ applications of the proposition, so we have a success probability of at least $(1-\eps)^{10}\geq 1-\eps'$.
\end{proof} 

\subsubsection{Proof in \ref{eq:case2}}
\begin{proof}[Proof of Proposition \ref{prop:prop1} in \ref{eq:case2}:]
Take $L\in2\N$ large enough for (\ref{eq:case2}) to hold and fix some large $t\in\N$. We introduce the two sites 
\[z^1\coloneqq \Big(L+n,\frac L2,...,\frac L2\Big)\quad\text{ and }\quad z^2\coloneqq \Big(0,\frac L2,...,\frac L2\Big).\]
In the case $d=1$ we read this as $z^1=L+n$ and $z^2=0$. On the event 
\begin{align*}
\big\{Z^{(D_n,S^2)}_L\text{ survives}\big\}
\end{align*}
we consider a random sequence $(v_k)_{k\in\N}$ of particles by choosing $v_k$ from $Z^{(D_n,S^2)}_L(tk)$ in some deterministic way, say by choosing the minimal element in the lexicographical order. This sequence enables us to make infinitely many trials to find a fully occupied box at the required position: 

For every $k$, denote by $\big(Z^k(s))_{s\geq tk}$ the process obtained by taking $v_k$ as the new root and considering only its descendants. We define random variables
\begin{align*}A^i_k&\coloneqq \1\big\{(z^i+D_n,S^2)\leq Z^k(t(k+1))\big\}\quad\text{ for }k\in\N,i\in\{0,1\}.\end{align*}
We want to give a lower bound for the probability of $\{A^1_k=1\}$ and $\{A^2_k=1\}$, so consider 
\[M(z)\coloneqq \min_{x\in\{-L,...,L\}^d}\big\{P_\omega\big((z+D_n,S^2)\leq Z_{\{-L,...,3L\}\times\{-L,...,L\}^{d-1}}^{\{x\}}(t)\big)\big\}.\]
Setting now 
\begin{equation}\label{eq:nicealpha}\alpha\coloneqq \min\{\E[M(z^1)],\E[M(z^2)]\}>0\end{equation}
we can choose $k$ large enough for $(1-\alpha)^k\leq \delta$. Finally $T:=kt$ and 
\begin{align*}
A^1:= \1\big\{Z_L^{(D_n,S^2)}\text{ survives}, A^1_j=1\text{ for some }k\leq j\leq 2k\}.
\end{align*}
Observe that
\[\{A^1=1\}\subseteq \left\{\begin{matrix}\exists x\in\{L+n,...,2L+n\}\times\{-L,...,L\}^{d-1}, t\in[T, 2T] \\\text{s.t.  }(x+D_n,S^2)\leq Z^{(D_n,S^2)}_{\{-L,...,3L\}\times\{-L,...,L\}^{d-1}}(t)\end{matrix}\right\}\]
and
\[\P(A^1=1)
\geq (1-2\delta)\big(1-(1-\alpha)^k\big)\geq 1-3\delta\geq 1-\eps.\]
\end{proof}

\begin{proof}[Proof of Proposition \ref{prop:prop2} in \ref{eq:case2}:]
For this we choose the same values of $L$ and $T$, and observe that by symmetry the value of $\alpha$ does not change when we flip the sign of any coordinate in $z^1=(z^1_1,...,z^1_d)$ or $z^2=(z^2_1,...,z^2_d)$. So we choose them in such a way that 
\[\sign z^i_j=-\sign y_j\quad\text{ for all }j=2,...,d\text{ and }i=1,2,\]
where $y$ appeared in the statement of Proposition \ref{prop:prop2}. Now define $(z^{(i)})_{i\in\N}$ by $z^{(1)}:=y+z^1$ and 
\[z^{(i)}:=y+z^1+\sum_{j=2}^i (-1)^{j}z^2\quad \text{ for }i\geq 2.\]
Note that we have chosen the signs in such a way that for all $i$ we have
\begin{align*}
z^{(i)}\in\{L+n\}\times\{-L,...,L\}^{d-1}.
\end{align*}
Let $\widetilde A^1_i$ be the same indicator function as $A^1_i$ with $z^1$ replaced by $z^{(1)}$, and let $\widetilde A$ be defined as $A$ with $A^1_i$ replaced by $\widetilde A^1_i$. On $\{\widetilde A=1\}$ we find a minimal $K_1\in \{k,...,2k\}$ such that $\widetilde A^1_{K_1}=1$. That is
\begin{equation}\label{eq:temp}
(z^{(1)}+D_n,S^2) \leq Z^{\{s\}\times(y+D_n,S^2)}_{\{-L,...,3L\}\times\{-3L,...,3L\}^{d-1}}(tK_1).
\end{equation}
We now have to improve (\ref{eq:temp}) so that it holds for some time in $[5T,...,6T]$. For this we define indicator functions
\[\widetilde B^i\coloneqq \1\big\{\exists \;j\in \{k,...2k\}\colon (z^{(i+1)}+D_n,S^2) \leq Z^{\,\{t^{(i)}\}\times(z^{(i)}+D_n,S^2)}_{\{-L,...,3L\}\times\{-3L,...,3L\}^{d-1}}(t^{(i)}+jt)\big\}.\]
So $\{\widetilde B^i=1\}$ is (up to shifts) the same event as $\{A^1=1\}$ with $z^1$ replaced by $z^2$ and started from $(z^{(i)}+D_n,S^2)$ at some time $t^{(i)}$, which we did not specify yet. Note that from our choice of $\alpha$ in \eqref{eq:nicealpha}, the same argument as before yields 
\[\P(\widetilde B^i=1)\geq 1-3\delta\geq 1-\eps\quad\text{ for all }i.\]
We now recursively define $(t^{(i)})_{i\in\N}$. Start from $t^{(1)}:=K_1t$, and assume we have found $t^{(1)},...,t^{(i)}$. On $\{\widetilde B^i=1\}$ we find a minimal value $K_{i+1}$ such that $z^{(i+1)}+D_n$ is occupied by at least $S^2$ particles at time $t^{(i)}+tK_{i+1}$. Then we proceed by
\[t^{(i+1)}\coloneqq t^{(i)}+K_{i+1}k.\]
Since $t^{(i+1)}-t^{(i)}\in[T,...,2T]$ we have
\[\big\{\widetilde A=\widetilde B_2=...=\widetilde B_6=1\big\}\subseteq A^{s, y}(L,T,n,S).\]
So the claim follows from our choice of $\eps$ in (\ref{eq:defeps}) and because the event on the left hand side has probability at least $(1-\eps)^6$.
\end{proof}

%

\begin{thebibliography}{10}

\bibitem{criticalcontact}
Carol Bezuidenhout and Geoffrey Grimmett, \emph{The critical contact process
  dies out}, The Annals of Probability (1990), 1462--1482.

\bibitem{majorization}
Richard~A. Brualdi and Geir Dahl, \emph{Majorization for partially ordered
  sets}, Discrete Mathematics \textbf{313} (2013), no.~22, 2592 -- 2601.

\bibitem{carmona}
Philippe Carmona and Yueyun Hu, \emph{Fluctuation exponents and large
  deviations for directed polymers in a random environment}, Stochastic
  Processes and their Applications \textbf{112} (2004), no.~2, 285--308.

\bibitem{yoshida_path}
Francis Comets and Nobuo Yoshida, \emph{Directed polymers in random environment
  are diffusive at weak disorder}, Annals of Probability \textbf{34} (2006),
  no.~5, 1746--1770.

\bibitem{yoshida}
\bysame, \emph{Branching random walks in space–time random environment:
  Survival probability, global and local growth rates}, Journal of Theoretical
  Probability \textbf{24} (2011), no.~3, 657--687.

\bibitem{ErhdenHollMaillard}
D.~Erhard, F.~den Hollander, and G.~Maillard, \emph{The parabolic {A}nderson
  model in a dynamic random environment: basic properties of the quenched
  {L}yapunov exponent}, Ann. Inst. Henri Poincar\'e Probab. Stat. \textbf{50}
  (2014), no.~4, 1231--1275. \MR{3269993}

\bibitem{shiga2}
Tasuku Furuoya and Tokuzo Shiga, \emph{{Sample Lyapunov exponent for a class of
  linear Markovian systems over $\Z^d$}}, Osaka Journal of Mathematics
  \textbf{35} (1998), 35--72.

\bibitem{garetmarchand}
Olivier Garet and Régine Marchand, \emph{The critical branching random walk in
  a random environment dies out}, Electronic Communications in Probability
  \textbf{18} (2013), no. 9, 1--15.

\bibitem{gaertnerdenHoll}
J.~G{\"a}rtner and F.~den Hollander, \emph{Intermittency in a catalytic random
  medium}, Ann. Probab. \textbf{34} (2006), no.~6, 2219--2287. \MR{2294981}

\bibitem{WolfgangK}
Wolfgang K{\"o}nig, \emph{The parabolic {A}nderson model}, Pathways in
  Mathematics, Birkh\"auser/Springer, [Cham], 2016, Random walk in random
  potential. \MR{3526112}

\bibitem{liggettIPS}
Thomas~M. Liggett, \emph{Interacting particle systems}, Grundlehren der
  Mathematischen Wissenschaften [Fundamental Principles of Mathematical
  Sciences], vol. 276, Springer-Verlag, New York, 1985. \MR{776231}

\bibitem{liggett}
\bysame, \emph{Stochastic interacting systems: Contact, voter and exclusion
  processes}, Die Grundlehren der mathematischen Wissenschaften in
  Einzeldarstellungen, Springer, 1999.

\bibitem{mullerstoyan}
Alfred M{\"u}ller and Dietrich Stoyan, \emph{Comparison methods for stochastic
  models and risks}, Wiley Series in Probability and Statistics, Wiley, 2002.

\bibitem{ortgieseroberts}
Marcel Ortgiese and Matthew~I. Roberts, \emph{Intermittency for branching
  random walk in {P}areto environment}, Ann. Probab. \textbf{44} (2016), no.~3,
  2198--2263. \MR{3502604}

\bibitem{shiga}
Tokuzo Shiga, \emph{Exponential decay rate of survival probability in a
  disastrous random environment}, Probability Theory and Related Fields
  \textbf{108} (1997), no.~3, 417--439.

\bibitem{smith}
Walter~L. Smith, \emph{Necessary conditions for almost sure extinction of a
  branching process with random environment}, The Annals of Mathematical
  Statistics \textbf{39} (1968), 2136--2140. \MR{0237006 (38 \#5299)}

\bibitem{tanny}
David Tanny, \emph{Limit theorems for branching processes in a random
  environment}, Annals of Probability \textbf{5} (1977), no.~1, 100--116.
  \MR{0426189 (54 \#14135)}

\end{thebibliography}


\providecommand{\bysame}{\leavevmode\hbox to3em{\hrulefill}\thinspace}
\providecommand{\MR}{\relax\ifhmode\unskip\space\fi MR }
\providecommand{\MRhref}[2]{%
  \href{http://www.ams.org/mathscinet-getitem?mr=#1}{#2}
}
\providecommand{\href}[2]{#2}


\ACKNO{We thank Nobuo Yoshida for hosting SJ, for many inspiring and fruitful discussions and for helpful comments on an earlier version of this paper. SJ thanks the DAAD for financial support in the PROMOS scholarship. Further, we are grateful to an anonymous referee for pointing out a mistake, and for many comments and suggestions.}


\end{document}